\documentclass[11pt,UKenglish]{amsart}
\usepackage{babel}
\usepackage[T1]{fontenc}
\usepackage[utf8]{inputenc}
\usepackage[babel]{microtype}
\usepackage{csquotes}
\usepackage{amscd, color}
\usepackage{amssymb}
\usepackage{amsthm}
\usepackage{caption}
\usepackage{subcaption}
\usepackage{thmtools}
\usepackage{thm-restate}
\usepackage{graphicx}
\usepackage[pdfusetitle]{hyperref}
\usepackage[breakable]{tcolorbox}
\usepackage{longtable}
\usepackage[mathscr]{euscript}
\usepackage{etoc}
\usepackage{cancel}
\usepackage{soul}

\hypersetup{pdfencoding=auto,bookmarksnumbered=true,hypertexnames=false}

\etocsettocstyle{\section*{\contentsname}\hypersetup{hidelinks}}{}
\etocsetstyle{section}{}{\noindent}{\textbf{\etocname{}}\dotfill\etocpage\smallskip\par}{}
\etocsetstyle{subsection}{\etocskipfirstprefix\leftskip2em\rightskip1.3em\noindent}{, }{\textit{\footnotesize\etocname{} (p. \etocpage)}}{.\par\leftskip0cm\rightskip0cm}

\DeclareCaptionLabelFormat{bf-parens}{\textbf{(#2)}}

\theoremstyle{plain}

\newtheorem{Thm}{Theorem}[section]
\newtheorem{Lemma}[Thm]{\bf Lemma}
\newtheorem{Corollary}[Thm]{\bf Corollary}
\newtheorem{Theorem}[Thm]{\bf Theorem}
\newtheorem{Proposition}[Thm]{\bf Proposition}
\newtheorem{Notation}[Thm]{\bf Notation}

\theoremstyle{definition}
\newtheorem{Definition}[Thm]{\bf Definition}
\theoremstyle{remark}
\newtheorem{Remark}[Thm]{\bf Remark}

\newtheoremstyle{Cl}{5pt}{3pt}{\sl}{}{\it}{:}{.5em}{}
\theoremstyle{Cl}

\newtcolorbox{bluebox}{colback=blue!5!white,colframe=blue!75!black,breakable,oversize}

\graphicspath{{img/}}

\DeclareMathOperator{\spt}{spt \,}

\def\al{\alpha}
\def\be{\beta}
\def\de{\delta}

\def\eps{\varepsilon}
\def\ga{\gamma}
\def\G{\Gamma}

\def\la{\lambda}

\def\om{\omega}

\def\bP{\mathbb P}
\def\R{{\mathbb R}}

\def\N{{\mathbb N}}
\def\T{{\mathbb T}}
\def\A{{\mathcal A}}

\def\cG{{\mathcal G}}

\def\cT{{\mathcal T}}

\def\ov{\overline}

\def\co{\mathrm{co}}

\def \txt{\qquad\text}

\def\begincproof{

    \begin{proof}
    }
    \def\endcproof{
    \end{proof}

}

\DeclareRobustCommand{\SkipTocEntry}[5]{}

\swapnumbers

\title[Mean field games]{Differential and  variational approach to first order Mean Field Games in a generalized form.}

\author{Antonio Siconolfi}
\address{Dipartimento di Matematica, Sapienza Universit\`a di Roma, Italy.}
\email{siconolf@mat.uniroma1.it}

\begin{document}

\maketitle

\begin{abstract}

We investigate time-dependent, first-order Mean Field Games on the torus $\T^N$
comparing, in a broad and general framework, the classical differential formulation
— given by a Hamilton–Jacobi equation coupled with a continuity equation — with a
variational approach based on fixed points of a multivalued map acting on probability
measures over trajectories.

We prove existence of fixed points for very general Hamiltonians. When the Hamiltonian
is differentiable with respect to the momentum, we show that the evaluation curve of
any such fixed point solves a continuity equation driven by a vector field $W_{g_0}$,
associated with the final condition $g_0$.

This field is defined without requiring additional regularity conditions on the value
function solving the Hamilton--Jacobi equation. The field coincides with the classical
vector field of Mean Field Systems at space--differentiability points of the value
function lying in the space--time regions where optimal trajectories concentrate.

Our analysis therefore provides a unified framework that bridges the differential
and variational viewpoints in Mean Field Games, showing how aggregate optimality
conditions naturally lead to continuity-equation descriptions under the sole
assumption of differentiability of $H$ in $p$.

\end{abstract}

\section{Introduction}

We study a time--dependent first order Mean Field Game (MFG) model posed on the
$N$--dimensional flat torus $\T^N$ over a finite time horizon $T$. The Hamiltonian
\[
H(x,\mu,p)
\qquad x \in \T^N, \ \mu \in \bP(\T^N), \ p \in \R^N,
\]
where $\bP(\T^N)$ is the space of Borel probability measures in $\T^N$, is considered
in a rather general form. As the title suggests, our main purpose is to compare
— within the broadest possible setting — two approaches to the problem, which we
refer to as the {\em differential} and the {\em variational} formulation.

The first one, classical in the MFG literature, consists in determining solutions
— precisely defined later — of a system comprising a backward Hamilton–Jacobi (HJ)
equation coupled with a forward continuity equation.

In the HJ equation the Hamiltonian appears as $H(x,\xi(t),Dv)$, thus depending on
time through a curve of measures $\xi(t)$ in $\bP(\T^N)$ defined in $[0,T]$. This
curve is, in turn, the unknown of the continuity equation, driven by the vector field
\[
H_p(x,\xi(t),-Dv(x,t)),
\]
where $H_p$ denotes the derivative of the Hamiltonian with respect to the momentum
variable, and $Dv$ is the spatial gradient of the value function $v$, given by
Lax--Oleinik (LO) formula, solving the HJ equation.

As a matter of fact, the term $H_p(x,\xi(t),-Dv(x,t))$ is elusive, since the
existence of $Dv$ is not in general guaranteed.

The second method relies on the observation that the above system represents the
same optimization problem at two different scales: that of the typical agent (the HJ
equation) and that of the global population (the continuity equation). This viewpoint
shifts the focus from the differential system to a direct variational characterization
of the optimal objects — curves and measures — yielding a simple and direct formulation.

The problem ultimately reduces to finding fixed points of a suitable multivalued map
defined in a space of probability measures supported on curves of $\T^N$. In this
framework, the system formulation becomes less central, since no explicit knowledge
of the value function solving the HJ equation is required.

In the differential approach one works with curves of measures in $\bP(\T^N)$, while
in the variational framework the relevant fixed points lie in the space $\bP(\G)$ of
probability measures on $\G$, where $\G$ denotes the space of continuous curves
$[0,T] \to \T^N$. To each $\xi \in \bP(\G)$ one associates a canonical curve of
measures
\[\xi(t) = \mathrm{ev}_t \# \xi \quad t \in [0,T], \qquad
\mathrm{ev}_t : \G \to \T^N,
\]
where $\#$ stands for the push--forward. We call $t \mapsto \xi(t)$ the
{\em evaluation curve} associated to $\xi$.

Under basic assumptions — continuity of $H$ in all arguments plus convexity and
superlinearity in $p$ ({\bf (H1)}--{\bf (H3)}) — we establish existence of such
fixed points.

Under the additional assumption of differentiability of $H$ in $p$ ({\bf (H4)}), we
prove that the evaluation curve of any fixed point in $\bP(\G)$ solves a continuity
equation. Strengthening the superlinearity condition ({\bf (H5)}) yields that the
evaluation curve is absolutely continuous with a suitable exponent.

In the results above we replace the classical driving field
\[
H_p(x,t,-Dv(x,t))
\]
whose definition requires differentiability of the value function — an assumption not
generally available without imposing restrictive conditions on $H$. This motivates
the expression {\em in a generalized form} in the title.

The replacement vector field is well defined on the whole $\T^N \times [0,T]$
under {\bf (H1)}--{\bf (H4)}, without any further regularity of the value function.
It coincides with $H_p(x,t,-Dv(x,t))$, wherever the spatial gradient of $v$ does
exist, in the space--time region occupied by optimal trajectories — an issue to which
we return later.

The differential MFG framework was introduced in the mathematical community in 2006
in the seminal works of Lasry and Lions \cite{LL106,LL206,LL07} and, simultaneously,
by Huang, Caines and Malham\'{e} \cite{HCM06,HCM107,HCM207}. Since then it has
generated a substantial body of literature — both qualitative and quantitative —
drawing from viscosity solutions, optimal transport, Wasserstein spaces, optimal
control, game theory, probability, and Fokker–Planck equations.

The differential MFG system captures the behavior of a large population of
indistinguishable and negligible agents. The HJ equation represents the optimal
control problem of a typical agent. The convexity of $H$, though arising in a
game-theoretic model, stems from the negligibility of individual players, which
rules out direct interactions; other agents affect the equation only through an
aggregate density representing a conjectured evolution of the total state.

On the other hand, the decisions of each agent induce a global motion of the system,
represented by a curve of measures solving the continuity equation. The coupling
arises because optimal individual trajectories are integral curves of the same
vector field driving the continuity equation.

The equilibrium state in fact is reached when the conjectured aggregate evolution
matches the one actually induced by the optimal strategies. Mathematically, this
corresponds to a pair $(v,\xi(t))$ of an individual value function solution to the
HJ equation (in viscosity sense) and a curve of measures solving the continuity
equation (in distributional sense).

A nontrivial conceptual shift occurs when moving from the differential to the
variational approach: one passes from curves of measures to measures on curves,
among which the fixed points must be found. Several works develop this point of
view \cite{Sa00, Sa16, Sa17, CC18, Sa19}.

We organize our analysis along two complementary lines:

\begin{itemize}
  \item[--] the derivation of aggregate optimality conditions and their
    differential counterpart, namely the continuity equation;
  \item[--] the fixed--point problem for the multivalued map that relates
    aggregate decision rules to optimal collective behavior.
\end{itemize}

For the first line of investigation, we do not yet impose any dependence of the
Hamiltonian on an aggregate evolution. We therefore work with an {\em abstract}
Hamiltonian $H_0(x,t,p)$. Given a continuous final condition $g_0$ in the
Hamilton--Jacobi equation, we select---among the curves of minimal action for the
corresponding Lagrangian $L_0$---a distinguished family of trajectories, called
$g_0$--optimal.

We therefore consider  the Borel probability measures on the space of curves
supported on $g_0$--optimal paths. We  call them
$g_0$--optimal as well. Although the definition is elementary, these measures exhibit
strong variational features: using an appropriate version of Kantorovich
duality, we show that they minimize the lifted action functional
(Proposition~\ref{optimum}), under assumptions {\bf (H1)}--{\bf (H3)}.

When the Hamiltonian is differentiable in $p$ ({\bf (H4)}), a continuity equation
emerges as a necessary differential condition for any $g_0$--optimal measure. In
this framework, the value function automatically admits distinguished
directional derivatives on a set
\[
\A_{g_0} \subset \T^N \times [0,T],
\]
which determines a Borel vector field $W_{g_0}$. The set $\A_{g_0}$ is large in
the sense that, for every $g_0$--optimal curve $\zeta$, the points
$(\zeta(t),t)$ belong to a subset of $\A_{g_0}$ for almost every $t$. Moreover,
whenever the value function is differentiable in space at such a point,
\[
W_{g_0}(x,t) = H_p(x,t,-Dv(x,t)).
\]

We prove that each $g_0$--optimal trajectory is an integral curve of $W_{g_0}$.
Consequently, the evaluation curve of any $g_0$--optimal measure solves the
continuity equation driven by $W_{g_0}$.

For the second line of investigation, we return to Hamiltonians of the form
$H(x,\xi(t),p)$, under assumptions {\bf (H1)}--{\bf (H3)}. Employing the
locally convex variant of Kakutani’s fixed point theorem
(see~\cite{CC18,Sa19}), together with topological arguments, a suitable
perturbation of the Lagrangian, and the variational characterization of
$g_0$--optimal measures, we establish the existence of fixed points.

 Since each
fixed point is itself $g_0$--optimal, the previous analysis applies: the
evaluation curve $\xi^*(t)$ of any fixed point $\xi^*$ satisfies the continuity
equation driven by $W_{g_0}$, where $W_{g_0}$ is constructed from the value
function solving the Hamilton--Jacobi equation with Hamiltonian depending on
$\xi^*(t)$.

We believe that the most significant outputs of this work are the following:

For the abstract Hamiltonian $H_0(x,t,p)$:
\begin{itemize}
  \item[--] Under assumptions {\bf (H1)}--{\bf (H4)}, the fact that the evaluation
  curve solves the continuity equation driven by $W_{g_0}$ is a necessary differential
  condition for $g_0$--optimality of a measure in $\bP(\G)$ (Theorem \ref{assone}).
\end{itemize}

For the Hamiltonian $H(x,\mu,p)$ of the MFG model:
\begin{itemize}
  \item[--] Existence under the sole assumptions {\bf (H1)}--{\bf (H3)} of a fixed
  point of the multivalued map $\cT : \bP(\G) \to \bP(\G)$ defined in \eqref{defTau}
  (Theorem \ref{mfg}).
  \item[--] Under assumptions {\bf (H1)}--{\bf (H4)}, any fixed point $\xi^*$,
  coupled with the solution of the HJ equation with Hamiltonian $H(x,\xi^*(t),p)$
  given by LO formula, solves the differential MFG system with driving vector field
  $W_{g(\cdot,\xi^*(T))}$ (Theorem \ref{teoremone}).
\end{itemize}

\medskip

\subsection{Outline of the paper}

Section \ref{preli} introduces notations, basic measure--theoretic results used
throughout the paper, and the assumptions {\bf (H1)}--{\bf (H3)} on the Hamiltonian.
Section \ref{LLOO} discusses properties of the action functional and the LO formula
providing the relevant solution to the time--dependent HJ equation coupled with final
datum $g_0$. The key notion of $g_0$--optimal curve (Definition \ref{optifinal})
is derived, along with a compactness result (Proposition \ref{compattone}).

Section \ref{optical} proves that $g_0$--optimal measures exist for any initial
distribution. Their intrinsic variational properties are given in Proposition
\ref{optimum}, and the compactness result for optimal curves is extended to measures
(Proposition \ref{compattonebis}). Section \ref{field} is devoted to the construction
of the vector field $W_{g_0}$, which is then related to $H^0_p(x,t,-Dv(x,t))$ in
Section \ref{comparozzi}.

Section \ref{continual} contains the main Theorem \ref{assone} on $g_0$--optimality
and continuity equation. A preliminary step is the proof that any $g_0$--optimal
curve is an integral curve of $W_{g_0}$.

In the final sections we return to the Hamiltonian $H(x,\mu,p)$, and establish
the fixed point result (Theorem \ref{mfg}, Section \ref{fissa}). Finally,
Section \ref{generale} deals with the differential MFG result in a generalized
form (Theorem \ref{teoremone}).

The four appendices contain a summary on disintegration theory, an analysis on the
function given by LO formula, topological facts on the strict topology in the
non--local compactness setting, and results on a perturbed Lagrangian used in the
fixed point theorem.

\bigskip

\section{Preliminaries}\label{preli}

\subsection{Notations and basic results}

We denote by $\T^N$ the $N$--dimensional flat {\em torus}. Given $x$, $y$ in $\T^N$,
we write $|x-y|$ for their Euclidean distance. By a {\em curve} in $\T^N$ we mean,
unless otherwise specified, an absolutely continuous (AC) curve.

In a space of measures, the term {\em curve} will instead refer to any map from a
real interval to the space.

For any set $E$, $\chi_E$ denotes the characteristic function.

We use the abbreviation lsc (resp. usc) for lower (resp. upper) semicontinuous.

For a measure $\mu$, we denote by $\spt \mu$ its {\em support}.

For $r \geq 1$, we denote by $L^r([0,T])$ the space of functions defined in $[0,T]$
which are $r$--summable with respect to the $1$--dimensional Lebesgue measure.

If $X$ is a metric space, $x \in X$ and $r > 0$, we denote by $B(x,r)$ the ball
centered at $x$ with radius $r$.

Given a continuous function $u$, we call {\em subtangent} (resp. {\em supertangent})
to $u$ at a point a viscosity test function from below (resp. from above). Unless
otherwise specified, the function $u$ and the subtangent (resp. supertangent)
coincide at the test point. If they coincide {\em only} at such a point, we say that
the subtangent (resp. supertangent) is {\em strict}.

We recall two standard results that will be used throughout the paper:

\begin{Proposition}\label{paglia}
Let $X$ be a Polish space, let $f : X \to \R$ be a lsc function, and let
$\mu_n \to \mu$ narrowly. Then
\[
\liminf_n \int f \, d\mu_n \geq \int f \, d \mu.
\]
\end{Proposition}

\begin{proof}
See \cite[Section 5.1.1]{AmS08}.
\end{proof}

\smallskip

\begin{Proposition}\label{change}
Let $X$, $Y$ be Polish spaces, $\mu$ a measure on $X$. We consider a Borel map
$G : X \to Y$ and a Borel function $\varphi : Y \to \R$ bounded from below. Then
\begin{equation}
\int \varphi(G(x)) \, d\mu = \int \varphi(y) \, d G \# \mu,
\end{equation}
where $G\# \mu$ is the push--forward of $\mu$ to $Y$ through $G$.
\end{Proposition}

\begin{proof}
See \cite[Proposition 8.1]{ADPM11}.
\end{proof}

\subsection{The measure--theoretic framework}

Given $T > 0$, we denote by $\G$ the space of continuous curves $[0,T] \to \T^N$,
endowed with the uniform topology. For any $\zeta \in \G$ we denote by $\|\zeta\|$
its uniform norm.

We denote by $\bP(\T^N)$ and $\bP(\G)$ the spaces of Borel probability measures
in $\T^N$ and $\G$, respectively, both endowed with the narrow topology.

\begin{Remark}\label{cedric}
All Wasserstein distances metrize the narrow topology on both $\bP(\T^N)$ and
$\bP(\G)$, making them Polish spaces. This follows from the compactness of $\T^N$
and the boundedness of $(\G,\|\cdot\|)$.
\end{Remark}

\smallskip

\begin{Notation}
We denote by $d_W$ the first Wasserstein distance both on $\bP(\T^N)$ and
$\bP(\G)$.
\end{Notation}

\smallskip

\begin{Notation}\label{postcedric}
For $\mu \in \bP(\T^N)$ and $f$ integrable with respect to $\mu$, we set
\[
\langle \mu,f\rangle = \int f \, d \mu.
\]
\end{Notation}

\smallskip

\begin{Definition}\label{spacetime}
For $\zeta \in \G$, we define the {\em space--time occupation measure} $\mu_\zeta$
on $\T^N \times [0,T]$ by
\[
\mu_\zeta(B)
= \frac{1}{T} \int_0^T \chi_{\{ t \mid (\zeta(t),t) \in B \}} \, dt
\]
for any Borel set $B \subset \T^N \times [0,T]$. Thus $\mu_{\zeta}(B)$ gives the
fraction of time that $(\zeta(t),t)$ spends in $B$.
\end{Definition}

\smallskip

For sets of the form $E \times I$, with $E \subset \T^N$ and $I$ an interval
contained in $[0,T]$,
\[
\mu_\zeta(E \times I)
= \frac{1}{T} \int_I \chi_{\{ t \mid \zeta(t) \in E \}} \, dt.
\]

The first marginal of $\mu_\zeta$ is the usual {\em (space) occupation measure}
\[
\mu_\zeta(A \times [0,T])
= \frac{1}{T} \int_0^T \chi_{\{ t \mid \zeta(t) \in A \}} \, dt
\txt{for $A \subset \T^N$ Borel.}
\]

\smallskip

For $\xi \in \bP(\G)$, we define the {\em evaluation curve}
\[
\xi(t) = \mathrm{ev}_t \# \xi
\txt{for $t \in [0,T]$,}
\]
where $\#$ indicates the push--forward, and
\[
\mathrm{ev}_t(\zeta) = \zeta(t)
\txt{for $\zeta \in \G$.}
\]

We also refer to $\xi(0)$ and $\xi(T)$ as the {\em initial} and {\em final
distribution} of $\xi$.

\medskip

\begin{Lemma}\label{preservone}
Let $\xi \in \bP(\G)$. Then the evaluation curve $t \mapsto \xi(t)$ is continuous
from $[0,T]$ with the natural topology to $\bP(\T^N)$ with the narrow topology.
\end{Lemma}

\begin{proof}
Given a sequence $t_n$ in $[0,T]$ converging to some $t_0$, we aim at showing that
\[
\int f \, d\xi(t_n) \to \int f \, d\xi(t_0)
\txt{for any continuous function $f : \T^N \to \R$.}
\]

By the change of variable formula in Proposition \ref{change} with
$X = \G$, $Y = \T^N$, $\mu = \xi$, $G = \mathrm{ev}_{t_n} / \mathrm{ev}_t$,
$\varphi = f$, we get
\[
\int f \, d\xi(t_n) = \int f(\zeta(t_n)) \, d\xi(\zeta)
\quad\txt{and}\quad
\int f \, d\xi(t_0) = \int f(\zeta(t_0)) \, d\xi(\zeta).
\]

Since
\[
f(\zeta(t_n)) \to f(\zeta(t_0))
\txt{for any $\zeta \in \G$,}
\]
we get by dominated convergence theorem
\[
\int f(\zeta(t_n)) \, d\xi \to \int f(\zeta(t_0)) \, d\xi,
\]
which concludes the proof.
\end{proof}

\smallskip

\begin{Lemma}\label{lemfranciosa1}
Let $\xi_n, \xi \in \bP(\G)$ with $\xi_n \to \xi$, then
\[
\xi_n(t) \to \xi(t)
\qquad
\hbox{in $\bP(\T^N)$, uniformly in $[0,T]$, with respect to $d_W$.}
\]
\end{Lemma}

\begin{proof}
We derive from the assumptions and Remark \ref{cedric} that $\xi_n$ converges to
$\xi$ with respect to $d_W$. There consequently exists a sequence of measures
$\ga_n$ in $\G \times \G$ with marginals $\xi_n$ and $\xi$ satisfying
\begin{equation}\label{francio1}
\int \|\eta_n - \eta\| \, d\ga_n(\eta_n,\eta) \to 0.
\end{equation}

We define, for $t \in [0,T]$, a map $\pi_t : \G \times \G \to \T^N \times \T^N$ as
\[
\pi_t(\zeta_1,\zeta_2) = (\zeta_1(t), \zeta_2(t)).
\]

Since for any $n$, $t$, and any Borel set $E \subset \T^N$,
\begin{align*}
\pi_t \# \ga_n(E \times \T^N)
&= \ga_n(\{ (\zeta,\eta) \mid \zeta(t) \in E,\, \eta(t) \in \T^N \}) \\
&= \xi_n(\{ \zeta \mid \zeta(t) \in E \})
= \xi_n(E),
\end{align*}
we derive that the first marginals of the $\pi_t \# \gamma_n$'s are $\xi_n(t)$,
and a similar computation yields that all the second marginals coincide with
$\xi(t)$.

We therefore get by change of variable formula \eqref{change}
\begin{equation}\label{franciosa2}
\int |x-y| \, d(\pi_t \# \gamma_n)
= \int |\eta_n(t)-\eta(t)| \, d\gamma_n(\eta_n,\eta)
\leq \int \|\eta_n-\eta\| \, d\gamma_n,
\end{equation}
which gives, in force of \eqref{francio1}, the claimed uniform convergence
of $\xi_n(t)$ to $\xi(t)$.
\end{proof}

\smallskip

\begin{Lemma}\label{lemdemure3}
Let $\mu_n \to \mu$ in $\bP(\T^N)$, and let $f_n$ be continuous functions in
$\T^N$ converging uniformly to $f$. Then
\[
\langle \mu_n,f_n \rangle \to \langle \mu,f \rangle.
\]
\end{Lemma}

\begin{proof}
Given $\eps > 0$, we obtain for $n$ large enough
\[
\int f_n \, d\mu_n \geq \int f \, d\mu_n - \eps,
\]
which implies
\[
\liminf_n \int f_n \, d\mu_n \geq \int f \, d\mu - \eps.
\]

The same argument, with obvious adaptations, gives a converse inequality
for the $\limsup$. The constant $\eps$ being arbitrary, the assertion follows.
\end{proof}

\medskip

\subsection{Assumptions on the Hamiltonians}\label{H0L0}

In Sections \ref{LLOO}--\ref{comparozzi}, we work with a non--autonomous
Hamiltonian
\[
H_0 : \T^N \times [0,T] \times \R^N \to \R,
\]
in Section \ref{fissa} (the fixed point argument), and Section \ref{generale}
(the generalized MFG system) time--dependence will be mediated by a curve of
probability measures on $\T^N$.

We assume that $H_0$ satisfies:

\begin{itemize}
  \item[{\bf (H1)}] continuity in all arguments;
  \item[{\bf (H2)}] convexity in $p$;
  \item[{\bf (H3)}] uniformly superlinearity in $p$ for $(x,t)$ varying in
  $\T^N \times [0,T]$;
\end{itemize}

These conditions will remain in force throughout the paper.

We denote by $L_0(x,t,q)$ the Lagrangian associated to $H_0$ via Fenchel
transform. Under {\bf (H1)}, {\bf (H2)}, {\bf (H3)} $L_0$ is continuous,
convex in the velocity variable $q$, and superlinear, uniformly for $(x,t)$
varying in $\T^N \times [0,T]$.

\smallskip

The superlinearity is equivalent to the existence of a function
$\theta : [0,+\infty) \to [0,+\infty)$ with
\begin{equation}\label{pettis}
L_0(x,t,q) \geq \theta(|q|)
\quad\hbox{for $(x,t) \in \T^N \times [0,T]$ and}\quad
\lim_{s \to +\infty} \frac{\theta(s)}{s} = +\infty.
\end{equation}

By the Corollary in \cite[Proposition 2.2.6]{Cla90} these assumptions imply
that the $H_0$ and $L_0$ are locally Lipschitz continuous in $p$ and $q$
uniformly in $(x,t)$.

\smallskip

\begin{Notation}\label{M0mL}
We set
\[
M_0 = \max \{ |L_0(x,t,0)| \mid (x,t) \in \T^N \times [0,T] \}, \quad
m_{L_0} = \min \{ L_0(x,t,q) \mid (x,t,q) \in \T^N \times [0,T] \times \R^N \}.
\]
\end{Notation}

\bigskip

\section{Minimal action and Lax--Oleinik formula}\label{LLOO}

\medskip

\subsection{Minimal action}

\begin{Definition}
Let $[t_1,t_2] \subset [0,T]$, and let $\zeta$ be a curve in $\T^N$ defined on
$[t_1,t_2]$, we define its {\em action} as
\begin{equation}\label{defAL}
\int_{t_1}^{t_2} L_0(\zeta,t,\dot\zeta) \, dt.
\end{equation}
If $\zeta$ is continuous, but not absolutely continuous, its action is set
to $+ \infty$.
\end{Definition}

\smallskip

To simplify notations, we write
\begin{equation}\label{defALbis}
A_{L_0}(\zeta) = \int_0^T L_0(\zeta,t,\dot\zeta) \, dt.
\end{equation}

\smallskip

\begin{Proposition}\label{sciare}
For any $a \in \R$, the set
\[
C = \{ \zeta \in \G \mid A_{L_0}(\zeta) \leq a \}
\]
is compact in the uniform topology.
\end{Proposition}

\begin{proof}
By \eqref{pettis} the curves in $C$ are equiabsolutely continuous, see
\cite[Theorem 2.12]{BuGH98}. Since they take values in the compact set $\T^N$,
they are also equibounded. The result follows by Ascoli--Arzel\`{a} theorem.
\end{proof}

\smallskip

\begin{Proposition}\label{presciare}
The functional
\[
\zeta \mapsto A_{L_0}(\zeta)
\]
is lsc (possibly $+\infty$) in $\G$ with respect to the uniform topology.
\end{Proposition}

\begin{proof}
This follows from \cite[Theorem 3.6]{BuGH98}, using the continuity of $L_0$
and its convexity in the velocity variable.
\end{proof}

\smallskip

Since any curve in $\G$ with finite action is AC, the Lebesgue differentiation
yields:

\begin{Proposition}\label{master}
Let $\zeta \in \G$ with $A_{L_0}(\zeta) < +\infty$, then the set
\[
C
=
\{ (x,t) \mid x = \zeta(t) \;\hbox{with}\;
t \;\hbox{diffe. point of}\; \zeta \;\hbox{and Lebesgue point of}\;
L_0(\zeta,t,\dot\zeta) \}
\]
satisfies $\mu_\zeta(C) = 1$.
\end{Proposition}

\smallskip

\begin{Definition}
A curve $\eta : [0,T] \to \T^N$ is said {\em of minimal action} in $[0,T]$ if
\[
A_{L_0}(\eta)
= \min \{ A_{L_0}(\zeta) \mid \zeta(0) = \eta(0), \ \zeta(T) = \eta(T) \}.
\]
\end{Definition}

\smallskip

As a consequence of Propositions \ref{sciare}, \ref{presciare}, we have:

\begin{Proposition}
For any $(x,y) \in \T^N \times \T^N$, there exist curves of minimal action
$\eta$ in $[0,T]$ with $\eta(0) = x$, $\eta(T) = y$.
\end{Proposition}

\smallskip

Given $x$, $y$ in $\T^N$, we define
\begin{equation}\label{defS}
S_T(x,y)
= \min \{ A_{L_0}(\zeta) \mid \zeta \ \hbox{with}\ \zeta(0) = x, \ \zeta(T) = y \}.
\end{equation}

\smallskip

\begin{Proposition}\label{conti}
The function
\[
(x,y) \mapsto S_T(x,y)
\]
is bounded and lsc.
\end{Proposition}

\begin{proof}
Fix $(x_0,y_0)$ in $\T^N \times \T^N$. Joining $x_0$ to $y_0$ by a geodesic
defined in $[0,T]$, we find
\[
S_T(x_0,y_0)
\leq \frac{|x_0-y_0|}{T} \,
\max \{ |L_0(x,t,q)| \mid |q| = 1 \}.
\]

From Notation \ref{M0mL} we also have
\[
S_T(x_0,y_0) \geq m_{L_0} \, T.
\]

Thus $S_T(\cdot,\cdot)$ is bounded.

Now let $(x_n,y_n) \to (x_0,y_0)$, and let $\zeta_n$ be minimizers realizing
$S_T(x_n,y_n)$. The corresponding actions $A_{L_0}(\zeta_n)$ are uniformly
bounded, hence by Propositions \ref{sciare}, \ref{presciare} the $\zeta_n$'s
converge, up to subsequences, uniformly to $\zeta$ joining $x_0$ to $y_0$. Then
\[
\liminf S_T(x_n,y_n)
= \liminf A_{L_0}(\zeta_n)
\geq A_{L_0}(\zeta)
\geq S_T(x_0,y_0).
\]
This proves lower semicontinuity.
\end{proof}

\medskip

\subsection{HJ equation}

We consider the time--dependent Hamilton--Jacobi (HJ) equation
\begin{equation}\label{HJ-}\tag{HJ}
- u_t + H_0(x,t,-Du) = 0 \quad\hbox{in $ \T^N \times (0,T)$,}
\end{equation}
where solutions are understood in the viscosity sense.

\smallskip

A continuous solution to \eqref{HJ-} with final condition $g_0$ is given by the
Lax--Oleinik (LO) formula, see Appendix \ref{LxO}:
\begin{equation}\label{LO}
v(x,t)
= \min_{\zeta(t)=x}
\left\{
g_0(\zeta(T))
+ \int_t^T L_0(\zeta,s,\dot\zeta) \, ds
\right\}
\qquad\hbox{for $(x,t) \in \T^N \times [0,T)$.}
\end{equation}
The minimum exists by Propositions \ref{sciare}, \ref{presciare} and the continuity
of $g_0$.

\smallskip

\begin{Remark}\label{barlone}
Under assumptions {\bf (H1)}--{\bf (H3)} equation \eqref{HJ-} does not admit a unique
continuous solution in $\T^N \times [0,T]$, even with a continuous final datum $g_0$.
Uniqueness may be enforced, for instance, by requiring
\[
|H_0(x,t,p) - H_0(y,t,p)|
\leq \om_0((1+|p|) |x-y|)
\]
for some modulus $\om_0 : [0,+\infty) \to [0,+\infty)$ and all $x$, $y$ in $\T^N$,
$t \in [0,T]$, $p \in \R^N$, see e.g. \cite[Theorem 5.1]{Ba13}.

Uniqueness is not needed here: we will work exclusively with the Lax–Oleinik
solution \eqref{LO}, which represents the value function of the associated optimal
control problem.
\end{Remark}

\smallskip

\begin{Notation}\label{vvv}
We denote by $v$ the LO solution defined in \eqref{LO}.
\end{Notation}

\begin{Lemma}\label{sub}
For any $x$, $y$ in $\T^N$, $t_2 > t_1 \in [0,T]$, and any curve $\zeta$
in $[t_1,t_2]$ linking $x$ to $y$,
\[
v(x,t_1) - v(y,t_2)
\leq \int_{t_1}^{t_2} L_0(\zeta,t,\dot\zeta) \, dt.
\]
\end{Lemma}

\smallskip

The choice of continuous final datum $g_0$ induces a natural decomposition of the
minimal action curves into (possibly overlapping) families.

\begin{Definition}\label{optifinal}
For a continuous function $g_0$ on $\T^N$, a curve $\zeta_0 : [0,T] \to \T^N$
is called {\it $g_0$--optimal} if
\[
g_0(\zeta_0(T)) + A_{L_0}(\zeta_0)
= \min \{ g_0(\zeta(T)) + A_{L_0}(\zeta) \mid
\zeta \in \G \ \hbox{with} \ \zeta(0) = \zeta_0(0) \}.
\]
Equivalently
\[
A_{L_0}(\zeta_0) = v(\zeta_0(0),0) - g_0(\zeta_0(T)).
\]
\end{Definition}

\smallskip

Existence of such curves comes from Propositions \ref{sciare}, \ref{presciare}
and the continuity of $g_0$.

\smallskip

\begin{Remark}\label{dinamo}
If $\zeta_0$ is $g_0$--optimal then for all $t \in [0,T)$
\[
v(\zeta_0(t),t)
= g_0(\zeta_0(T)) + \int_t^T L_0(\zeta_0,s,\dot\zeta_0) \, ds \txt{.}
\]
\end{Remark}

\smallskip

\begin{Proposition}\label{compattone}
Let $M > 0$ and let $\om$ be a modulus. Then the set of curves
$[0,T] \to \T^N$ that are $g_0$--optimal for some continuous function $g_0$
with
\begin{equation}\label{compattone0}
\max_{\T^N} |g_0(x)| \leq M,
\qquad
|g_0(x_1) - g_0(x_2)| \leq \om(|x_1-x_2|)
\quad\hbox{for any $x_1$, $x_2$ in $\T^N$}
\end{equation}
is compact in $\G$.
\end{Proposition}

\begin{proof}
Let $\zeta_n$ be $g^0_n$--optimal curves for final data $g^0_n$ satisfying
\eqref{compattone0}. We have
\begin{equation}\label{compattone1}
v_n(\zeta_n(0),0) = g^0_n(\zeta_n(T)) + A_{L_0}(\zeta_n),
\end{equation}
where $v_n$ is the LO solution to \eqref{HJ-} agreeing with $g^0_n$ at $t=T$.
By \eqref{compattone0} the sequence $g^0_n$ uniformly converges in $\T^N$, up
to a subsequence, to a continuous function $g_0$ still satisfying
\eqref{compattone0}. Consequently, by Proposition \ref{stabia}, the solutions
$v_n$ uniformly converge in $\T^N \times [0,T]$, up to subsequences, to the
LO solution $v$ of \eqref{HJ-} with final datum $g_0$.

Up to extracting a subsequence, we therefore have that the $v_n$'s are
equibounded in $\T^N \times [0,T]$. This implies by \eqref{compattone1}
that the actions $A_{L_0}(\zeta_n)$ are equibounded as well. We can therefore
extract, by Proposition \ref{sciare}, a subsequence uniformly converging in
$[0,T]$ to a limit curve $\zeta_0$. Passing to the limit, as $n \to +\infty$,
in \eqref{compattone1} and taking into account Proposition \ref{presciare}
we get
\[
v(\zeta_0(0),0) - g_0(\zeta_0(T)) \geq A_{L_0}(\zeta_0),
\]
which implies by Lemma \ref{sub} that $\zeta_0$ is $g_0$--optimal. This
concludes the argument.
\end{proof}

\bigskip

\section{Optimal measures}\label{optical}

We introduce a notion parallel to Definition \ref{optifinal} for measures in
$\bP(\G)$:

\begin{Definition}\label{moptifinal}
Given a continuous function $g_0$ in $\T^N$, a measure $\xi \in \bP(\G)$ is
called {\it $g_0$--optimal} if it is supported on $g_0$--optimal curves.
\end{Definition}

\medskip

\subsection{Existence of $g_0$--optimal measures}\label{esiotti}

\begin{Proposition}\label{esisto}
For any continuous function $g_0$ and $\mu \in \bP(\T^N)$, there exist
$g_0$--optimal measures with initial distribution $\mu$.
\end{Proposition}

\smallskip

This existence result follows from measurable selection principles for
multivalued maps together with the disintegration Theorem
(Appendix \ref{disintegra0}).

\smallskip

We begin by defining the multivalued map $\cG_{g_0} : \T^N \to \G$ as
\begin{equation}\label{curvacce0}
\cG_{g_0}(x) = \{ \hbox{$g_0$--optimal curves starting at $x$} \}.
\end{equation}

\smallskip

By disintegration theory any $g_0$--optimal measure — if it exists — is, roughly
speaking, parametrized by the measurable families $\{ \eta_x \} \subset \bP(\G)$
with $\spt(\eta_x) \subset \cG_{g_0}(x)$, for each $x \in \T^N$. See Definition
\ref{misuramisura} for terminology.

\smallskip

More precisely:

\begin{Proposition}\label{ottimissimo}
Given any continuous function $g_0$, the existence of a $g_0$--optimal measure,
for any initial distribution $\mu$, is equivalent to the existence of a measurable
family $\{ \eta_x \} \subset \bP(\G)$ such that
\[
\spt \eta_x \subset \cG_{g_0}(x)
\txt{for $x$ varying in $\T^N$.}
\]
\end{Proposition}

\begin{proof}
Assume that such a family $\{ \eta_x \}$ exists, and define
\[
\xi(B) = \int \eta_x(B) \, d\mu(x)
\qquad\hbox{for any Borel subset $B$ in $\G$,}
\]
then $\xi$ is a probability measure on $\G$ with $\xi(0) = \mu$ by Proposition
\ref{dis}, and
\[
\spt \xi \subset \overline{ \cup_x \spt \eta_x },
\]
so $\xi$ is supported on $g_0$--optimal curves.

Conversely, if $\xi$ is a measure as in the statement, its disintegration
via $\mathrm{ev}_0$ (Theorem \ref{disbis}) yields the required family.
\end{proof}

To construct such families $\{ \eta_x \}$, we need some preliminary material.

\smallskip

\begin{Definition}\label{uffa}
Let $X$, $Y$ be Polish spaces and $\cG : X \to Y$ a multivalued map. We say
that it is {\em measurable} if
\begin{itemize}
  \item[--] it has closed values;
  \item[--] for any closed subset $E$ of $Y$, the set
  \[
  \cG^-(E) := \{ x \in X \mid \cG(x) \cap E \neq \emptyset \}
  \]
  is Borel.
\end{itemize}
\end{Definition}

\begin{Proposition}\label{uffaz}
Every measurable multivalued map $\cG : X \to Y$ admits a measurable selection,
namely a measurable function $h : X \to Y$ with
\[
h(x) \in \cG(x)
\qquad\hbox{for any $x$.}
\]
\end{Proposition}

\begin{proof}
See \cite[Theorem 3.6]{CV77}.
\end{proof}

\smallskip

\begin{Lemma}\label{dead}
The multivalued map $\cG_{g_0}$, defined in \eqref{curvacce0}, is measurable
for any continuous final datum $g_0$.
\end{Lemma}

\begin{proof}
It is clear that $\cG_{g_0}$ is closed valued.

Given a closed subset $E \subset \G$, let $x_n$ be a sequence in $\cG_{g_0}^-(E)$
converging to $x_0 \in \T^N$, and let $\zeta_n \in \cG_{g_0}(x_n) \cap E$.
The curves $\zeta_n$ then satisfy
\begin{equation}\label{dead2}
v(x_n,0) = g_0(\zeta_n(T)) + A_{L_0}(\zeta_n).
\end{equation}

Since the action of the $\zeta_n$'s is bounded, they uniformly converge,
up to a subsequence, by Proposition \ref{sciare} to a curve linking $x_0$
to some $y_0$, which is optimal by Proposition \ref{presciare}. Passing to the
limit, as $n$ goes to infinity, in \eqref{dead2} we get that $\cG_g^-(E)$
is closed, which yields the assertion.
\end{proof}

\smallskip

\begin{proof}[{\it Proof of Proposition \ref{esisto}}]
Using Proposition \ref{ottimissimo} it suffices to show the existence of a
family of measures $\{ \eta_x \}$ in $\bP(\G)$ as in the statement of
Proposition \ref{ottimissimo}.

By Lemma \ref{dead} and Proposition \ref{uffaz} there exists a measurable
selection $h(x) \in \cG_{g_0}(x)$. Define
\[
\eta_x = \delta_{h(x)} \in \bP(\G),
\]
where $\de$ stands for the Dirac measure. This family is measurable
(Definition \ref{misuramisura}) because $h$ is so. Since $h$ is a selection
of $\cG_{g_0}$, each $\eta_{x}$ is supported by $g_0$--optimal curves
starting at $x$.
\end{proof}

\medskip

\subsection{Intrinsic characterization of $g_0$--optimal measures}

We now provide intrinsic characterizations of $g_0$--optimal measures by lifting
both the action functional to $\bP(\G)$ and the cost $S_T$ to
$\bP(\T^N\times\T^N)$.

\smallskip

Some preliminary notions are needed:

\begin{Definition}
Given $\mu$, $\nu$ in $\bP(\T^N)$, we call {\em transport plan} between them
any Borel probability measure in $\T^N \times \T^N$ with first marginal $\mu$
and second marginal $\nu$.
\end{Definition}

\smallskip

We define
\[
S_T(\mu,\nu)
= \inf_\ga \int S_T(x,y) \, d\ga
\txt{$\mu$, $\nu$ in $\bP(\T^N)$,}
\]
where the infimum is taken over all the transport plans $\ga$ between $\mu$
and $\nu$.

\smallskip

\begin{Proposition}\label{buono}
For any $\mu, \nu \in \bP(\T^N)$ there exist transport plans achieving the infimum
in the definition of $S_T(\mu,\nu)$.
\end{Proposition}

\begin{proof}
Since the cost $S_T$ is lsc (Proposition \ref{conti}), the functional
\[
\ga \mapsto \int S_T(x,y) \, d\ga
\]
is lsc in the narrow topology by Proposition \ref{paglia}. Since the family
of transport plans between $\mu$ and $\nu$ is compact, minimizers exist.
\end{proof}

\smallskip

We derive from Lemma \ref{sub}

\begin{Proposition}\label{narduz}
Let $\mu, \nu \in \bP(\T^N)$, and let $v$ be the LO solution of \eqref{HJ-}
with final datum $g_0$. Then
\[
S_T(\mu,\nu) \geq \langle \mu, v(\cdot,0) \rangle - \langle \nu, g_0 \rangle.
\]
\end{Proposition}

\begin{proof}
Let $\ga$ be a transport plan between $\mu$ and $\nu$. By Lemma \ref{sub}
\[
\int S_T(x,y) \, d\ga
\geq \int [v(x,0) - g_0(y)] \, d\ga
= \langle \mu, v(\cdot,0) \rangle - \langle \nu, g_0 \rangle.
\]
Hence the claim.
\end{proof}

\smallskip

By Proposition \ref{presciare} and the coercivity of $L_0$, we have:

\begin{Proposition}\label{semi}
The action functional
\[
\zeta \mapsto A_{L_0}(\zeta)
\quad \zeta \in \G
\]
is integrable (with possibly infinite integral) against any measure in
$\bP(\G)$.
\end{Proposition}

\smallskip

Thus, for $\xi \in \bP(\Gamma)$, we may define its action as
\[
\int A_{L_0}(\zeta) \, d\xi(\zeta).
\]

\smallskip

Next, we relate transport plans and the action of measures in $\bP(\G)$:

\begin{Proposition}\label{pizza}
For any $\xi \in \bP(\Gamma)$
\[
\int A_{L_0}(\zeta) \, d\xi \geq S_T(\xi(0),\xi(T)).
\]
\end{Proposition}

\begin{proof}
Let $\eta_{(x,y)}$ be the disintegration of $\xi$ with respect to
\begin{equation}\label{defF}
\mathrm{ev}_{(0,T)} : \Gamma \to \T^N \times \T^N,
\qquad
\mathrm{ev}_{(0,T)}(\zeta) = (\zeta(0),\zeta(T)).
\end{equation}

Then $\mathrm{ev}_{(0,T)} \# \xi$ is a transport plan between $\xi(0)$ and
$\xi(T)$, and
\begin{align}\label{preoptimum}
\int A_{L_0}(\zeta)\, d\xi
&= \int \Bigl( \int A_{L_0}(\zeta) \, d\eta_{(x,y)} \Bigr)
  d(\mathrm{ev}_{(0,T)} \# \xi)(x,y) \\
&\geq \int S_T(x,y) \, d(\mathrm{ev}_{(0,T)} \# \xi)(x,y)
 \geq S_T(\xi(0),\xi(T)), \nonumber
\end{align}
where the leftmost equality is obtained by the very definition of disintegration.
\end{proof}

\smallskip

We need Kantorovitch Theorem in the following form:

\begin{Theorem}\label{kanton}
Given $\mu, \nu \in \bP(\T^N)$, there exists a continuous function
$g_0 : \T^N \to \R$ such that the LO solution $v$ of \eqref{HJ-}, with final
datum $g_0$, satisfies
\[
S_T(\mu,\nu) = \langle \mu, v(\cdot,0) \rangle - \langle \nu, g_0 \rangle.
\]
\end{Theorem}

\begin{proof}
We adapt the strategy of \cite[Theorem 1.1]{Sic23}, where the cost is assumed
continuous. In our setting $S_T$ is only lsc (Proposition \ref{conti}). Thus,
we approximate it monotonically from below by continuous costs, according to
\cite[Theorem 1.3]{V03}.

Following \cite[Theorem 1.1]{Sic23}, the result is first proved under the
assumption that $\mu$, $\nu$ are finite convex combinations of Dirac masses.
In that case the continuity of the approximating costs allows constructing
the appropriate final datum $g_0$ (\cite[Proposition 6.1]{Sic23}).
The conclusion is then obtained by a finite--dimensional argument based on
Hahn--Banach Theorem applied to a suitable space of measures.

One passes to general measures by exploiting the density of the convex hull
of Dirac measures in $\bP(\T^N)$ with respect to the narrow topology, and
the stability properties of LO solutions (Proposition \ref{stabia}).
\end{proof}

\smallskip

\begin{Proposition}\label{optimum}
For a measure $\xi \in \bP(\G)$, the following statements are equivalent:
\begin{itemize}
  \item[(i)] $\xi$ is $g_0$--optimal for some continuous function $g_0$;
  \item[(ii)] $\displaystyle \int A_{L_0}(\zeta) \, d\xi = \langle \xi(0), v(\cdot,0) \rangle - \langle \xi(T), g_0 \rangle$.
\end{itemize}
\end{Proposition}

\begin{proof}
Assume $\xi$ satisfies (i), and let $v$ be the LO solution of \eqref{HJ-} with
final datum $g_0$. By disintegration of $\xi$ with respect to $\mathrm{ev}_{(0,T)}$,
denoted by $\eta_{(x,y)}$, we get
\begin{equation}\label{materiotto1}
\int A_{L_0}(\zeta) \, d\xi
= \int \Bigl( \int A_{L_0}(\zeta) \, d\eta_{(x,y)} \Bigr)
  d(\mathrm{ev}_{(0,T)} \# \xi)(x,y).
\end{equation}

Since $\xi$ is $g_0$--optimal, all curves $\zeta$ in its support are
$g_0$--optimal, hence for $(x,y)=(\zeta(0),\zeta(T))$ we have
\[
A_{L_0}(\zeta) = v(x,0)-g_0(y).
\]

Therefore by \eqref{materiotto1}
\[
\int A_{L_0}(\zeta) \, d\xi(\zeta)
= \int [v(x,0)-g_0(y)] \, d(\mathrm{ev}_{(0,T)} \# \xi)(x,y)
= \langle \xi(0), v(\cdot,0) \rangle - \langle \xi(T), g_0 \rangle.
\]

Conversely, Kantorovitch duality Theorem (Theorem \ref{kanton}) yields
\begin{equation}\label{materiotto2}
S_T(\xi(0),\xi(T))
= \langle \xi(0), v(\cdot,0)\rangle - \langle \xi(T), g_0 \rangle
\end{equation}
for some continuous function $g_0$ and the LO solution $v$ of \eqref{HJ-}
with final condition $g_0$.

Assume that $\xi$ does not satisfy (i) for this $g_0$, then there exists a
set of curves with nonvanishing $\xi$--measure on which the inequality in
Lemma \ref{sub} is strict. Equivalently the open set
\[
E = \{ \zeta \in \G \mid A_{L_0}(\zeta) > v(\zeta(0),0) -g_0(\zeta(T)) \}
\]
has positive $\xi$--measure.  By integrating, we get
\[\int A_{L_0}(\zeta) \, d\xi > \langle \xi(0), v(\cdot,0) \rangle - \langle \xi(T), g_0 \rangle\]
which contradicts (ii).
\end{proof}

\smallskip

The following result directly comes from Proposition \ref{compattone}. It will
be used in the proof of the fixed point theorem of Section \ref{fissa}.

\smallskip

\begin{Proposition}\label{compattonebis}
Let $M > 0$ and $\om$ be a modulus. The set of $g_0$--optimal measures in
$\bP(\G)$, where $g_0$ ranges over continuous functions with
\begin{equation}\label{cicciuzzo}
\max_{\T^N} |g_0(x)| \leq M,
\qquad
|g_0(x_1) - g_0(x_2)| \leq \om(|x_2-x_1|)
\quad\hbox{for any $x_1$, $x_2$ in $\T^N$}
\end{equation}
is a compact subset in $\bP(\G)$.
\end{Proposition}

\begin{proof}
By Proposition \ref{compattone} all such measures are supported on a compact
subset of $\G$. Compactness then follows from Prohorov's theorem.
\end{proof}

\bigskip

\section{A distinguished Borel vector field}\label{field}

We introduce the following additional assumption on $H_0$:

\begin{itemize}
  \item[{\bf (H4)}] the function $p \mapsto H_0(x,t,p)$ is differentiable in $\R^N$,
\end{itemize}

which is equivalent to the strict convexity of the Lagrangian:

\begin{equation}\label{H10bis}
\hbox{$q \mapsto L_0(x,t,q)$ is strictly convex.}
\end{equation}

\smallskip

\begin{Notation}
We denote by $H^0_p$ the derivative of $H_0$ with respect to $p$.
\end{Notation}

\smallskip

We fix a continuous final datum $g_0$ and denote by $v$ the LO solution to \eqref{HJ-}
with final datum $g_0$.
In accordance with the title, we construct in this section, under assumptions
{\bf (H1)}--{\bf (H4)}, a Borel vector field in $\T^N \times [0,T]$, depending on
$g_0$ and $v$, which will be linked to $g_0$--optimal trajectories in Section
\ref{continual} and employed in a generalized formulation of MFG systems in
Section \ref{generale}.

\smallskip

\begin{Notation}
For $(x,t) \in \T^N \times (0,T)$ and $q \in \R^N$, we denote by
\[
D_{(q,1)} v(x,t)
\]
the directional derivative of $v$ at $(x,t)$ in the space--time direction $(q,1)$,
whenever it exists, namely
\[
D_{(q,1)} v(x,t) = \lim_{h \to 0} \frac{v(x+hq, t+h) - v(x,t)}{h}.
\]
\end{Notation}

\smallskip

\begin{Lemma}\label{prealana}
For any $(x,t,q) \in \T^N \times (0,T) \times \R^N$ we have
\begin{equation}\label{prealana1}
 \liminf_{h \to 0} \frac{v(x+hq,t+h) - v(x,t)}{h} \geq -L_0(x,t,q).
\end{equation}
\end{Lemma}

\begin{proof}
For $h < 0$, Lemma \ref{sub} yields
\[
v(x+hq,t+h) - v(x,t) \leq \int_{t+h}^t L_0(x+(\tau-t)q,\tau,q)\, d\tau,
\]
which gives
\begin{equation}\label{alana1}
\frac{v(x+hq,t+h) - v(x,t)}{h}
\geq -\frac{1}{h} \int_{t+h}^t -L_0(x+(\tau-t)q, \tau, q)\, d\tau.
\end{equation}

For $h > 0$, a similar argument gives
\[
v(x,t) - v(x+hq,t+h) \leq \int_t^{t+h} L_0(x+(\tau-t)q,\tau,q)\, d\tau,
\]
and
\begin{equation}\label{alana2}
\frac{v(x+hq,t+h) - v(x,t)}{h}
\geq \frac{1}{h} \int_t^{t+h} -L_0(x+(\tau-t)q,\tau,q)\, d\tau.
\end{equation}

The claim follows by combining \eqref{alana1} and \eqref{alana2}.
\end{proof}

\smallskip

We define the multivalued (possibly empty--valued) map
$\T^N \times (0,T) \to \R^N$
\begin{equation}\label{defcW}
\mathcal W_{g_0}(x,t)
= \left\{ q \in \R^N \mid D_{(q,1)} v(x,t) = -L_0(x,t,q) \right\}.
\end{equation}

We then set
\begin{equation}\label{defcA}
\A_{g_0} = \{ (x,t) \mid \mathcal W_{g_0}(x,t) \neq \emptyset \}.
\end{equation}

\smallskip

\begin{Remark}
Since the LO solution $v(x,t)$ is, in general, only continuous, the set $\A_{g_0}$
might a priori be small.
However, we shall prove in Proposition \ref{teogigionzi} that $\A_{g_0}$ contains
a full measure set (with respect to $\mu_\zeta$) for any $g_0$--optimal
$\zeta \in \G$.
In particular, for every optimal curve $\zeta$, $(\zeta(t),t) \in \A_{g_0}$ for
almost all times.
\end{Remark}

\smallskip

\begin{Proposition}\label{salgado}
For any $(x,t) \in \T^N \times (0,T)$, the set $\mathcal W_{g_0}(x,t)$ contains
at most one element.
\end{Proposition}

\begin{proof}
By \eqref{H10bis} $L_0(x,t,\cdot)$ is strictly convex.
Assume by contradiction that $q_1 \neq q_2 \in \mathcal W_{g_0}(x,t)$.

Let $h > 0$, then by definition
\begin{align*}
v(x+h q_1, t+h) &= v(x,t) - h L_0(x,t,q_1) + \mathrm o(h), \\
v(x-h q_2, t-h) &= v(x,t) + h L_0(x,t,q_2) + \mathrm o(h),
\end{align*}
where $\mathrm o(h)$ denotes the Landau symbol.
Subtracting yields
\[
v(x - h q_2, t-h) - v(x+h q_1, t+h)
= 2h \left( \frac{L_0(x,t,q_2)}{2} + \frac{L_0(x,t,q_1)}{2} \right)
+ \mathrm o(h),
\]
\begin{equation}\label{salgado1}
\lim_{h \to 0}
\frac{1}{2h} \left( v(x - h q_2, t-h) - v(x+h q_1, t+h) \right)
=
\frac{L_0(x,t,q_2)}{2} + \frac{L_0(x,t,q_1)}{2}.
\end{equation}

We define in $[t-h,t+h]$ the curves
\[
\zeta_h(t) = y_h + t\left( \frac{q_1}{2} + \frac{q_2}{2} \right),
\]
with
\[
y_h = x + h q_1 - (t+h)\left( \frac{q_1}{2} + \frac{q_2}{2} \right).
\]

We have
\[
\zeta_h(t-h) = x - h q_2,
\qquad
\zeta_h(t+h) = x + h q_1.
\]

Lemma \ref{sub} yields
\begin{align*}
\frac{1}{2h} \bigl( v(x - h q_2, t-h) - v(x + h q_1, t+h) \bigr)
&\leq \frac{1}{2h} \int_{t-h}^{t+h}
L_0(\zeta_h, \tau, \dot\zeta_h)\, d\tau \\
&= L_0\!\left( \zeta_h(\tau_h), \tau_h,
\frac{q_1}{2} + \frac{q_2}{2} \right)
\end{align*}
for a suitable $\tau_h \in [t-h,t+h]$.

Since
\[
\lim_{h \to 0} \zeta_h(\tau_h) = x,
\qquad
\lim_{h \to 0} \tau_h = t,
\]
we obtain
\[
\lim_{h \to 0}
\frac{1}{2h} \bigl( v(x - h q_2, t-h) - v(x + h q_1, t+h) \bigr)
\leq
L_0\!\left( x, t, \frac{q_1}{2} + \frac{q_2}{2} \right).
\]

Comparing with \eqref{salgado1}, we obtain
\[
\frac{1}{2}\, L_0(x,t,q_1)
+
\frac{1}{2}\, L_0(x,t,q_2)
\leq
L_0\!\left( x, t, \frac{q_1}{2} + \frac{q_2}{2} \right),
\]
contradicting the strict convexity of $L_0$.
\end{proof}

\smallskip

\begin{Notation}\label{doppiav}
We denote by $W_{g_0}$ the map which associates to any
$(x,t) \in \A_{g_0}$ the unique vector in $\mathcal W_{g_0}(x,t)$.
\end{Notation}

\medskip

We next set
\[
G(x,t,q)
=
\limsup_{h \to 0}
\frac{v(x+h q,t+h) - v(x,t)}{h},
\qquad
(x,t,q) \in \T^N \times [0,T] \times \R^N.
\]

\smallskip

By Lemma \ref{prealana} we obtain:

\begin{Corollary}\label{alana}
A triple $(x,t,q)$ satisfies
\[
G(x,t,q) \leq -L_0(x,t,q)
\]
if and only if $(x,t) \in \A_{g_0}$ and $q = W_{g_0}(x,t)$.
\end{Corollary}

\smallskip

\begin{Lemma}\label{keiko}
The function $G : \T^N \times (0,T) \times \R^N \to \R$ is Borel.
\end{Lemma}

\begin{proof}
Let $J \subset (-1,1)$ be a countable dense subset not containing $0$.

Since the function
\[
h \mapsto \frac{v(x+h q, t+h) - v(x,t)}{h}
\]
is continuous for $h \neq 0$, we have
\begin{equation}\label{keiko1}
G(x,t,q)
=
\inf_{n \in \N}
\sup
\left\{
\frac{v(x+h q, t+h) - v(x,t)}{h}
\; \middle|\;
|h| < \frac{1}{n},\ h \in J
\right\}.
\end{equation}

The map
\[
(x,t,q)
\mapsto
\sup
\left\{
\frac{v(x+h q, t+h) - v(x,t)}{h}
\; \middle|\;
|h| < \frac{1}{n},\ h \in J
\right\}
\]
is Borel, being the supremum of countably many continuous functions.
Thus $G$ is the infimum of a countable family of Borel functions,
hence Borel.
\end{proof}

\smallskip

We recall classical notions (see \cite{Sri98}).
In what follows, $X$ and $Y$ denote arbitrary Polish spaces.

\begin{Definition}\label{sriana}
A subset $A \subset X$ is called \emph{analytic} if it is the projection of a Borel set
$B \subset X \times Y$ for some $Y$.
Any Borel subset of $X$ is analytic, but the converse is not true.
\end{Definition}

\begin{Definition}\label{srimap}
Given any subset $C \subset X$, a set $E \subset C$ is called Borel in $C$ if
\[
E = B \cap C
\]
for some Borel subset $B$ of $X$.

A map $f : C \to Y$ is called \emph{Borel} if for any Borel set $B_0$ of $Y$ there exists
a Borel set $B \subset X$ such that
\[
f^{-1}(B_0) = B \cap C.
\]
\end{Definition}

\smallskip

We give the following results in a form adapted to our setting.

\smallskip

\begin{Theorem}[Graph Theorem]\label{grapho}
Let $A \subset \T^N \times [0,T]$ be analytic.
A map $f : A \to \R^N$ is Borel if and only if its graph is a Borel subset of
$A \times \R^N$.
\end{Theorem}

\begin{proof}
See \cite[Theorem 4.5.2]{Sri98}.
\end{proof}

\smallskip

\begin{Theorem}[Extension Theorem]\label{extension}
Any Borel map $f : C \to Y$, with $C \subset X$, extends to a Borel map defined
on all of $X$.
\end{Theorem}

\begin{proof}
See \cite[Theorem 3.2.3]{Sri98}.
\end{proof}

\smallskip

\begin{Lemma}\label{trieste}
The set $\A_{g_0} \subset \T^N \times (0,T)$ is analytic, and
$W_{g_0}$ is a Borel map in the sense of Definition \ref{srimap}.
\end{Lemma}

\begin{proof}
From Corollary \ref{alana} we have
\[
\mathrm{graph}\, W_{g_0}
=
\{ (x,t,q) \in \T^N \times [0,T] \times \R^N
\mid G(x,t,q) + L_0(x,t,q) \leq 0\}.
\]

Since $G$ and $L_0$ are Borel, the graph is Borel in
$\T^N \times (0,T) \times \R^N$.
Thus $\A_{g_0}$, its projection, is analytic.
The claim follows from Theorem \ref{grapho}.
\end{proof}

\smallskip

We directly derive from Lemma \ref{trieste} and Theorem \ref{extension}:

\begin{Proposition}\label{borel}
The map $W_{g_0} : \A_{g_0} \to \R^N$ can be extended to a Borel map defined on
all of $\T^N \times [0,T]$.
\end{Proposition}

\smallskip

\begin{Notation}
We continue to denote the extended map by $W_{g_0}$.
\end{Notation}

\smallskip

If the LO solution $v$ is differentiable in space and time, we straightforwardly get:

\begin{Proposition}\label{sabbath}
If $v$ is differentiable in space and time at $(x_0,t_0)$ then
$(x_0,t_0) \in \A_{g_0}$ and
\[
H^0_p(x_0,t_0,-Dv(x_0,t_0)) = W_{g_0}(x_0,t_0).
\]
\end{Proposition}

\begin{proof}
It is immediate that $(x_0,t_0) \in \A_{g_0}$.

We have
\[
L_0(x_0,t_0,H^0_p(x_0,t_0,-Dv(x_0,t_0)))
=
- Dv(x_0,t_0)\cdot H^0_p(x_0,t_0,-Dv(x_0,t_0))
- v_t(x_0,t_0),
\]
and consequently
\begin{align*}
D_{(H^0_p(x_0,t_0,-Dv(x_0,t_0)),1)} v(x_0,t_0)
&= v_t(x_0,t_0)
  + Dv(x_0,t_0)\cdot H^0_p(x_0,t_0,-Dv(x_0,t_0)) \\
&= -L_0(x_0,t_0,H^0_p(x_0,t_0,-Dv(x_0,t_0))),
\end{align*}
which shows the assertion.
\end{proof}

\smallskip

If $v$ is differentiable only in space, the analysis requires more care and will be
carried out in Section \ref{comparozzi}.

\bigskip

\section{$g_0$--optimality and continuity equation}\label{continual}

We introduce a strengthened superlinearity assumption;

\begin{itemize}
  \item[{\bf (H5)}] There exists  $\eps_0 >0$ such that
  \[\lim_{|p| \to + \infty} \frac{H_0(x,t,p)}{|p|^{1+\eps_0}} = + \infty\]
  uniformly for $(x,t)$ varying in $\T^N \times [0,T]$.
\end{itemize}

This is equivalent to

\[  \lim_{|q| \to + \infty} \frac{L_0(x,t,q)}{|q|^{(1+\eps_0)^*}} = + \infty
  \txt{uniformly for $(x,t)$ varying in $\T^N \times [0,T]$,}
\]
where $(1+\eps_0)^*$ is the conjugate exponent of $1+\eps_0$. In particular,
this implies

\begin{equation}\label{H11bis}
L_0(x,t,q) \geq a + b|q|^{(1+\eps_0)^*}
\end{equation}
for suitable positive constants $a$,$b$, and all $(x,t,q) \in \T^N \times [0,T] \times \R^N$.

\smallskip

We recall two notions.

\begin{Definition}\label{precontinua}
 We say that a curve of measures $\mu(t)$, $\mu:[0,T] \to \bP(\T^N)$,  is {\em $(1+\eps_0)^*$--absolutely continuous} if
 \[d_W(\mu(t_1),\mu(t_2)) \leq \int_{t_1}^{t_2} \varphi(t) \, dt\]
 for a  function $\varphi \in L^{(1+\eps_0)^*}([0,T])$ and any $t_1 < t_2 \in [0,T]$.
\end{Definition}

\smallskip

\begin{Definition}\label{continua}
 A narrowly continuous  curve of measures $\mu(\cdot):[0,T] \to \bP(\T^N)$   is a solution  to the {\em continuity equation}  driven by $W_{g_0}$ in $[0,T]$
\begin{equation}\label{equ2}
 \frac \partial{\partial t} \mu(t) + \nabla \cdot (\mu(t) \, W_{g_0}(x,t) ) =0  \txt{in $\T^N \times [0,T]$}
\end{equation}
if  the following chain rule holds
\begin{equation}\label{equ1}
\int_0^T \langle \mu(t),  \varphi_t(\cdot, t) + D \varphi(\cdot,t) \cdot   W_{g_0}(\cdot,t) \rangle \, dt = \langle \mu(T),\varphi(\cdot ,T) \rangle - \langle \mu(0), \varphi(\cdot, 0) \rangle
\end{equation}
for every function $\varphi: \T^N \times [0,T] \to \R$ of class $C^1$.
\end{Definition}

\smallskip

 In this section we prove:

 \begin{Theorem}\label{assone}  Let $g_0$ be a continuous final datum, under the assumptions   {\bf(H1)}--{\bf(H4)}   any $g_0$--optimal measure $\xi \in \bP(\G)$ satisfies:
 \begin{itemize}
   \item[(i)] the evaluation curve $\xi(t)$ is a solution of the continuity equation driven by $W_{g_0}$.
   \end{itemize}
   If  in addition {\bf (H5)}  holds, then
   \begin{itemize}
   \item[(ii)] $\xi(t)$ is $(1+\eps_0)^*$--absolutely continuous.
 \end{itemize}
 \end{Theorem}

 \smallskip

 We start with some preliminaries:

 \begin{Lemma}\label{mowgli} Let $\zeta$, $h_n$ be a curve defined in $[0,T]$ and an infinitesimal negative sequence, respectively. Assume that $\zeta$ is differentiable at some $t_0 \in (0,T)$, then there exists $\rho_n >0$ with  $\lim_n \frac{\rho_n}{h_n} = 0$
 and
 \[ \frac{|\zeta(t_0) +( h_n - \rho_n)  \dot \zeta(t_0) - \zeta(t_0+h_n - \rho_n)|}{\rho_n} \leq 1\]
 for all $n$ sufficiently large .

 If instead $h_n$ is infinitesimal and positive, then there exists $\rho_n >0$ with  $\lim_n \frac{\rho_n}{h_n} = 0$
 and
 \[ \frac{|\zeta(t_0) +( h_n + \rho_n)  \dot \zeta(t_0) - \zeta(t_0+h_n + \rho_n)|}{\rho_n} \leq 1\]
 for all $n$ sufficiently large.
 \end{Lemma}
 \begin{proof} We only treat the case $h_n <0$, the other one being analogous.  For each  $n$, define
 \[f_n:\rho \mapsto \frac{|\zeta(t_0) +( h_n - \rho)  \dot \zeta(t_0) - \zeta(t_0+h_n - \rho)|}{\rho}\]
in the interval $(0,|h_n|]$. Since $\zeta$ is differentiable at $t_0$,
one has $\lim_n f_n(|h_n|)=0$,
hence
\begin{equation}\label{mowgli0}
  f_n(|h_n|) <1 \txt{for all large $n$ .}
\end{equation}
 For such an $n$, either
 \begin{equation}\label{mowgli1}
   \sup_{(0,|h_n|]} f_n >1
 \end{equation}
in which case, by continuity, there exists $\rho_n$ with $f_n(\rho_n)=1$, or else
\begin{equation}\label{mowgli2}
\sup_{(0,|h_n|]} f_n \leq 1,
\end{equation}
in  which case we  simply take $\rho_n = h_n^2$.  If for a  subsequence $n_k$ \eqref{mowgli2} holds then the stated inequality  is proved for $\rho_{n_k}$. If instead there exists a subsequence  $n_k$ satisfying \eqref{mowgli1}, then   by the differentiability of $\zeta$ at $t_0$ we  have
\[\lim_{n_k}  \frac{\rho_{n_k}}{h_{n_k} - \rho_{n_k}} = \lim_{n_k} \frac{|\zeta(t_0) +( h_{n_k} - \rho_{n_k})  \dot \zeta(t_0) - \zeta(t_0+h_{n_k} - \rho_{n_k})|}{h_{n_k} - \rho_{n_k}} =0 \]
and consequently\[ \lim_{n_k} \frac{h_{n_k}}{\rho_{n_k}}= \lim \frac{h_{n_k}-\rho_{n_k}}{\rho_{n_k}} + \lim \frac{\rho_{n_k}}{\rho_{n_k}} = -\infty,\]
which  completes the proof.
 \end{proof}

\smallskip

\begin{Lemma}\label{terme}
Let $\zeta:[0,T] \to \T^N$ be a $g_0$--optimal curve, and assume that
$t_0 \in (0,T)$ is a differentiability point of both $\zeta$ and of the
function $t \mapsto v(\zeta(t),t)$, with
\[
\frac d{dt} v(\zeta(t_0),t_0)
= - L_0(\zeta(t_0), t_0,\dot\zeta(t_0)).
\]
Then
\[
(\zeta(t_0),t_0) \in \A_{g_0}
\quad\text{and}\quad
\dot\zeta(t_0) = W_{g_0}(\zeta(t_0),t_0).
\]
\end{Lemma}

\begin{proof}
According to Corollary \ref{alana}, it is enough to show that
\begin{eqnarray}
  \limsup_{h \to 0} \frac {v(\zeta(t_0) + h \dot\zeta(t_0),t_0+h) -v(\zeta(t_0),t_0)}h &\leq& \lim_{h \to 0} \frac{v(\zeta(t_0+ h), t_0+h) - v(\zeta(t_0), t_0)}h  \nonumber \\&=& - L_0(\zeta(t_0),t_0,\dot \zeta(t_0)). \label{terme1}
 \end{eqnarray}

Let $(h_n)$ be any infinitesimal sequence such that
\[
\lim_{n\to\infty}
\frac{v(\zeta(t_0)+h_n\dot\zeta(t_0),t_0+h_n)-v(\zeta(t_0),t_0)}{h_n}
=
\limsup_{h\to 0}
\frac{v(\zeta(t_0)+h\dot\zeta(t_0),t_0+h)-v(\zeta(t_0),t_0)}{h}.
\]
Up to subsequences, we may assume that either all $h_n<0$ or all $h_n>0$.

 Case 1: \(h_n<0\)

\smallskip
Assume first that all $h_n<0$.
If \eqref{terme1} fails, then there exists $a>0$ such that for $n$ large,
\begin{equation}\label{terme00}
v(\zeta(t_0)+h_n\dot\zeta(t_0),t_0+h_n)
\;\le\;
v(\zeta(t_0+h_n),t_0+h_n) + a h_n.
\end{equation}

Let $(\rho_n)$ be the positive sequence associated with $(h_n)$ by
Lemma~\ref{mowgli}.
Since $\zeta$ is $g_0$--optimal,
\[
v(\zeta(t_0+h_n-\rho_n),t_0+h_n-\rho_n)
-
v(\zeta(t_0+h_n),t_0+h_n)
=
\int_{t_0+h_n-\rho_n}^{t_0+h_n}
L_0(\zeta,t,\dot\zeta)\,dt
\ge m_{L_0}\rho_n.
\]
Thus by \eqref{terme00},
\begin{equation}\label{terme0}
v(\zeta(t_0)+h_n\dot\zeta(t_0),t_0+h_n)
\le
v(\zeta(t_0+h_n-\rho_n),t_0+h_n-\rho_n)
+ a h_n - m_{L_0}\rho_n.
\end{equation}

For $n$ large, let $I_n=[t_0+h_n-\rho_n,\; t_0+h_n]$, and define the linear
interpolation
\[
\zeta_n(t)
=
\left(1 - \frac{t-(t_0+h_n-\rho_n)}{\rho_n}\right)
\zeta(t_0+h_n-\rho_n)
+
\frac{t-(t_0+h_n-\rho_n)}{\rho_n}
\bigl(\zeta(t_0) + h_n\dot\zeta(t_0)\bigr),
\qquad t\in I_n.
\]

Lemma~\ref{mowgli} gives
\[
|\dot\zeta_n(t)|
\le |\dot\zeta(t_0)| + 1
\qquad\text{for all }t\in I_n,n\in\N.
\]
Hence
\begin{equation}\label{terme3}
v(\zeta(t_0+h_n-\rho_n),t_0+h_n-\rho_n)
-
v(\zeta(t_0)+h_n\dot\zeta(t_0),t_0+h_n)
\le
\int_{I_n} L_0(\zeta_n,t,\dot\zeta_n)\,dt
\le b\rho_n
\end{equation}
for some constant $b>0$.

Combining \eqref{terme0} and \eqref{terme3}, we obtain
\[
-a h_n + m_{L_0}\rho_n
\le
b\rho_n,
\]
which is impossible because $h_n<0$ and
\(\rho_n/h_n \to 0\).
Thus \eqref{terme1} holds when $h_n<0$.

\smallskip

 Case 2: \(h_n>0\)

The argument is analogous.
If \eqref{terme1} fails, then for some $a>0$ and $n$ large,
\[
v(\zeta(t_0)+h_n\dot\zeta(t_0),t_0+h_n)
\ge
v(\zeta(t_0+h_n),t_0+h_n) + a h_n.
\]

Let \(\rho_n>0\) be the sequence given by Lemma~\ref{mowgli} for $h_n>0$.
Since $\zeta$ is $g_0$--optimal,
\[
v(\zeta(t_0+h_n),t_0+h_n)
-
v(\zeta(t_0+h_n+\rho_n),t_0+h_n+\rho_n)
=
\int_{t_0+h_n}^{t_0+h_n+\rho_n}
L_0(\zeta,t,\dot\zeta)\,dt
\ge m_{L_0}\rho_n.
\]
Hence
\begin{equation}\label{terme21}
v(\zeta(t_0)+h_n\dot\zeta(t_0),t_0+h_n)
\ge
v(\zeta(t_0+h_n+\rho_n),t_0+h_n+\rho_n)
+ a h_n + m_{L_0}\rho_n.
\end{equation}

Let $I_n=[t_0+h_n,\; t_0+h_n+\rho_n]$, and define the linear interpolation
\[
\zeta_n(t)
=
\left(1-\frac{t-(t_0+h_n)}{\rho_n}\right)
\bigl(\zeta(t_0)+h_n\dot\zeta(t_0)\bigr)
+
\frac{t-(t_0+h_n)}{\rho_n}
\zeta(t_0+h_n+\rho_n),
\qquad t\in I_n.
\]

As before,
\[
|\dot\zeta_n(t)| \le |\dot\zeta(t_0)|+1,
\quad t\in I_n,
\]
and therefore,
\begin{equation}\label{terme23}
v(\zeta(t_0)+h_n\dot\zeta(t_0),t_0+h_n)
-
v(\zeta(t_0+h_n+\rho_n),t_0+h_n+\rho_n)
\le
\int_{I_n} L_0(\zeta_n,t,\dot\zeta_n)\,dt
\le b\rho_n.
\end{equation}

Combining \eqref{terme21} and \eqref{terme23} yields
\[
a h_n + m_{L_0}\rho_n \le b\rho_n,
\]
which is impossible since \(h_n>0\) and \(\rho_n/h_n\to 0\).

\smallskip

Thus in both cases \eqref{terme1} holds, and the lemma follows.
\end{proof}

\smallskip

The next propositions shows, among other things, that for any $x_0 \in \T^N$  the vector field $W_{g_0}$ has an integral curve $\zeta$ with $\zeta(0)=x_0$.

 \smallskip

\begin{Proposition}\label{simple}  Any $g_0$--optimal curve  $\zeta$   is  an integral curve of $W_{g_0}$ with
\begin{equation}\label{simple00}
 (\zeta(t),t) \in \A_{g_0} \txt{for a.e. $t \in [0,T]$.}
\end{equation}
\end{Proposition}
\begin{proof}  By the dynamical programming principle, see Remark \ref{dinamo}, for any $t \in (0,T)$, $h$ small in modulus
\begin{eqnarray}
  v(\zeta(t ),t) &=& v(\zeta(t+h),t+h) + \int_{t}^{t+h} L_0(\zeta(\tau),\tau,\dot \zeta(\tau) \, d\tau  \txt{if $h >0$,}
   \label{simple001} \\
  v(\zeta(t+h ),t+h) &=& v(\zeta(t),t) + \int_{t+h}^{t} L_0(\zeta(\tau),\tau,\dot \zeta(\tau) \, d\tau  \txt{if $h <0$.} \label{simple002}
\end{eqnarray}
Let   $E_\zeta$ be the set   of the differentiability times  of $\zeta$,  which are also Lebesgue  points  of $t \mapsto  L_0(\zeta(t), t, \dot\zeta(t ))$.  Since $\zeta$  is AC and $t \mapsto  L_0(\zeta(t), t, \dot\zeta(t ))$ summable, $E_\zeta$ has  full  measure in $[0,T]$.

For $t \in E _\zeta$, we derive from \eqref{simple001}, \eqref{simple002}
\begin{eqnarray*}
 \lim_{h \to 0^+} \frac {v(\zeta(t+h),t+h)- v(\zeta(t),t)}h & = & \lim_{h \to 0^+}  - \frac  1h \, \int_{t}^{t+h} L_0(\zeta(\tau),\tau,\dot\zeta(\tau)) \, d \tau = - L_0(\zeta(t),t,\dot\zeta(t))\\
 \lim_{h \to 0^-} \frac {v(\zeta(t+h),t+h)- v(\zeta(t),t)}h  &=&  \lim_{h \to 0^-}   \frac  1h \, \int_{t+h}^{t} L_0(\zeta(\tau),\tau,\dot\zeta(\tau)) \, d \tau = - L_0(\zeta(t),t,\dot\zeta(t)).
\end{eqnarray*}
Hence $\zeta$ is differentiable at $t$ and
 \[ \frac d{dt} v(\zeta(t),t)= -L_0(\zeta(t),t,\dot\zeta(t)).\]
The claim   then follows from Lemma \ref{terme}.
\end{proof}

\smallskip

\begin{Corollary}\label{pregigionzi} Let $\zeta$ be a $g_0$--optimal curve and $t_0 \in (0,T)$ a differentiability point of $\zeta$ and a Lebesgue point of $t \mapsto L_0(\zeta(t),t,\dot\zeta(t))$. Then
\[(\zeta(t_0),t_0) \in \A_{g_0} \txt{and} \qquad \dot\zeta(t_0)= W_{g_0}(\zeta(t_0),t_0).\]
\end{Corollary}

\smallskip

We derive from Proposition \ref{simple}  and the definition of $g_0$--optimal measure:

\begin{Corollary}\label{velleda} Any $g_0$--optimal measure in $\bP(\G)$ is supported by integral curves of $W_{g_0}$.
\end{Corollary}

\smallskip

\begin{proof}[{\it Proof of Theorem \ref{assone}}]
  Let  $\varphi$ be a $C^1$ function defined in  $\T^N \times [0,T]$.  By the change of variable formula in Proposition \ref{change},  for each  $t \in [0,T]$
\begin{eqnarray}
   &&\int  [\varphi_t(x,t) + D \varphi(x,t) \cdot W_{g_0}(x,t)] \, d \xi(t) = \label{ritalma1} \\
   &&  \int [\varphi_t(\zeta(t),t) + D \varphi(\zeta(t),t) \cdot W_{g_0}(\zeta(t),t)] \, d\xi. \nonumber
\end{eqnarray}
The functions
\[ (\zeta,t) \mapsto  \zeta(t), \quad (\zeta,t) \mapsto \varphi_t(\zeta(t),t), \quad (\zeta,t) \mapsto  D\varphi(\zeta(t),t)\]
are  continuous and
\[(x,t) \mapsto W_{g_0}(x,t)\]
is Borel, hence the integrand on the right--side of \eqref{ritalma1} is Borel.   By Fubini's Theorem and \eqref{ritalma1}
\begin{eqnarray}
   &&\int_0^T \langle \xi(t), \varphi_t(\cdot,t) + D \varphi(\cdot,t) \cdot W_{g_0}(\cdot,t) \rangle \, dt = \label{ritalma2} \\
   && \int \left [ \int_0^T  \big [ \varphi_t(\zeta(t),t) + D \varphi(\zeta(t),t) \cdot W_{g_0}(\zeta(t),t) \big ] \, dt \right ] \, d\xi.  \nonumber
\end{eqnarray}
Since $\xi$ is supported on integral curves of $W_{g_0}$ (Corollary \ref{velleda}), the rightmost term in the above formula is equal to
\begin{equation}\label{ritalma3}
  \int [ \varphi(\zeta(T),T) - \varphi(\zeta(0),0) ]\, d\xi,
\end{equation}
and  by the change of variable formula
\begin{eqnarray}
  \int  \varphi(\zeta(T),T) \, d\xi &=&  \langle \xi(T), \varphi(\cdot,T) \rangle   \label{ritalma4}\\
  \int  \varphi(\zeta(0),0) \, d\xi &=&  \langle \xi(0), \varphi(\cdot,0). \rangle \label{ritalma5}
\end{eqnarray}
By combining \eqref{ritalma2}, \eqref{ritalma3}, \eqref{ritalma4}, \eqref{ritalma5}, we finally obtain
\[\int_0^T \langle \xi(t), \varphi_t(\cdot,t) + D \varphi(\cdot,t) \cdot W_{g_0}(\cdot,t) \rangle \, dt = \langle \xi(T), \varphi(\cdot,T) \rangle - \langle \xi(0), \varphi(\cdot,0) \rangle\]
which yields (i). We pass to  item (ii). By \cite[Theorem 8.3.1]{AmS08}  it is enough to show  that
\[t \mapsto  \langle \xi(t), |W_{g_0}(t,\cdot)|^{(1+\eps_0)^*}\rangle\]
belongs to $L^1([0,T])$. Using Fubini's Theorem, the $g_0$--optimality of $\xi$ and the growth bound \eqref{H11bis}, we compute
\begin{eqnarray*}
  \int_0^T \langle \xi(t) , |W_{g_0}(t,\cdot)|^{(1+\eps_0)^*}\rangle \, dt  &\leq& \int_0^T \langle \xi(t), c + c_0 L_0(x,t,W_{g_0}(x,t)) \, dt \\
  &=& \int \, \int_0^T \big (c + c_0 L_0(\zeta(t),t,W_{g_0}(\zeta(t),t) \big ) dt \, d\xi\\ &\leq & c \, T + c_0 \, \mathrm{osc} (v),
\end{eqnarray*}
for suitable positive constants $c$, $c_0$, where $\mathrm{osc}$ stands for oscillation and $v$ is the LO solution with final datum $g_0$. This shows that the function above is integrable, and the claim follows.
\end{proof}

\bigskip

\section{Comparison between $W_{g_0}$ and $H^0_p(x,t,-Dv(x,t))$}\label{comparozzi}

Exploiting some results from the previous section, we now establish a connection between $W_{g_0}$ and $H^0_p(x,t,-Dv(x,t))$, just assuming space--differentiability of $v$ (see also Proposition \ref{sabbath}). The main result of the section is:

\smallskip

\begin{Proposition}\label{teogigionzi}
There exists $B \subset \T^N \times [0,T]$ which contains a subset of $\mu_\zeta$--full measure, for any $g_0$--optimal curve $\zeta$, such that
\[
\A_{g_0} \supset B
\quad\hbox{and}\quad
W_{g_0}(x,t) = H^0_p(x,t,-Dv(x,t))
\]
for all $(x,t) \in B$ where $v$ is differentiable in space.
\end{Proposition}

\smallskip

We first prove a lemma on the extrema of functions of the form
\begin{equation}\label{funzione}
F^\varphi_h(x)
= \varphi(x) - h \, L_0\left(\zeta(t_0),t_0, \frac{\zeta(t_0+h)-x}{h}\right),
\end{equation}
where $h \in \R$ has sufficiently small modulus, $\zeta$ is a $g_0$--optimal curve, $t_0$ a fixed time in $(0,T)$, and $\varphi$ is a $C^1$ function on $\T^N$. This lemma will be used in the proof of Proposition \ref{teogigionzi}, and in a corollary, taking as $\varphi$ a $C^1$ subtangent/supertangent of $v(\cdot,t_0)$ at a space--differentiability point $(x_0,t_0)$. Recall that a function differentiable at a point admits both a $C^1$ subtangent and a $C^1$ supertangent at that point (see \cite[Lemma 1.8]{BICD97}).

\smallskip

\begin{Lemma}\label{gigiosi}
Let $\zeta$, $t_0$, $\varphi$, $h$ be as above, and let $F^\varphi_h$ be defined by \eqref{funzione}.
If $x_h$ is an extremal point of $F_h^\varphi$ in $\T^N$ (i.e., a local/global maximizer/minimizer), for any $h$, then
\begin{itemize}
  \item[{\bf (i)}] $\displaystyle \lim_{h\to 0} \frac{\zeta(t_0+h)-x_h}{h}
  = H^0_p\big(\zeta(t_0),t_0,-D\varphi(\zeta(t_0))\big)$;
  \item[{\bf (ii)}] $-D \varphi(\zeta(t_0)) \in \dfrac{\partial}{\partial q}
  L_0\big(\zeta(t_0),t_0,H^0_p(\zeta(t_0),t_0,-D\varphi(\zeta(t_0)))\big)$,
  where $\dfrac{\partial}{\partial q} L_0$ denotes the subdifferential of $L_0$ with respect to the velocity variable.
\end{itemize}
\end{Lemma}

\begin{proof}
Since $F^\varphi_h$ is continuous, it admits extremal points over $\T^N$ and at any such point $x_h$ the first--order optimality condition yields
\begin{equation}\label{gigiosi1}
D \varphi(x_h) + \frac{\partial}{\partial q} L_0\left(\zeta(t_0),t_0,\frac{\zeta(t_0+h) - x_h}{h}\right) \ni 0.
\end{equation}
Because $D\varphi(x)$ is bounded and $L_0$ is superlinear in $q$, it follows that
\[
\frac{\zeta(t_0+h) - x_h}{h}
\]
is bounded, hence $x_h \to \zeta(t_0)$ as $h \to 0$. Passing to the limit in \eqref{gigiosi1} and using closedness of the graph of the subdifferential, we obtain
\begin{equation}\label{tenta1}
D \varphi(\zeta(t_0)) + \frac{\partial}{\partial q} L_0(\zeta(t_0),t_0,q_0)\ni 0,
\end{equation}
for any limit point $q_0$ of $\dfrac{\zeta(t_0+h) - x_h}{h}$.
Since $L_0(\zeta(t_0),t_0,\cdot)$ is strictly convex, its subdifferential is strictly monotone, hence all such limit points must coincide, and
\[
\lim_{h \to 0} \frac {\zeta(t_0+h)- x_h}{h}
\]
exists. By duality it equals
\[
H^0_p(\zeta(t_0),t_0,-D\varphi(\zeta(t_0))),
\]
which proves {\bf (i)}. Item {\bf (ii)} follows directly from \eqref{tenta1}.
\end{proof}

\smallskip

\begin{proof}[{\it Proof of Proposition \ref{teogigionzi}}]
Fix a $g_0$--optimal curve $\zeta$. Assume that $v$ is differentiable in space at a point $(x_0,t_0) \in \T^N \times (0,T)$ with $x_0=\zeta(t_0)$. Assume further that $\zeta$ is differentiable at $t_0$ and that $t_0$ is a Lebesgue point for
\[
t \mapsto L_0(\zeta(t),t,\dot\zeta(t)).
\]
Let $\varphi: \T^N \to \R$ be a $C^1$ subtangent to $v(\cdot,t_0)$ at $\zeta(t_0)$ and, for each $h >0$, let $x_h$ be a maximizer of the function $F^\varphi_h$ defined in \eqref{funzione}. The gist of the proof is to show that
\[
\lim_{h\to 0} \frac{x_h-x_0}{h}=0.
\]
From Lemma \ref{gigiosi} {\bf (i)} $\dfrac{\zeta(t_0+h)-x_h}{h}$ stays bounded for $h \to 0$ and $x_h \to x_0$. Moreover, $\zeta$ is differentiable at $t_0$, and $t_0$ is a Lebesgue point for $L_0(\zeta(t),t,\dot\zeta(t))$. Using these facts, we can find a modulus $\om$ with
\begin{align}
\left|L_0\left(x_0,t_0, \frac{\zeta(t_0+h)-x_h}{h} \right)
- L_0\left(x_h + (t-t_0)\frac{\zeta(t_0+h)-x_h}{h}, t, \frac{\zeta(t_0+h)-x_h}{h} \right)\right|
< \om(h) \label{teogigionzi1}
\end{align}
for $t \in [t_0,t_0+h]$, and moreover
\begin{align}
\left| L_0\left(x_0,t_0, \frac{\zeta(t_0+h)-x_0}{h} \right) -  L_0(x_0,t_0,\dot\zeta(t_0)) \right|
&< \om(h), \label{teogigionzi2}\\
\left| \frac{1}{h} \int_{t_0}^{t_0+h} L_0(\zeta,t,\dot\zeta) \, dt - L_0(x_0,t_0,\dot\zeta(t_0)) \right|
&< \om(h). \label{teogigionzi3}
\end{align}
Define
\[
\eta_h(t)=  x_h + (t-t_0) \frac{\zeta(t_0+h)-x_h}{h},
\]
which is a curve joining $x_h$ to $\zeta(t_0+h)$ for $t$ varying in $[t_0,t_0+h]$.
Using \eqref{teogigionzi1}, and the fact that $\varphi$ is subtangent to $v(\cdot,t_0)$, we infer
\begin{align*}
F^\varphi_h(x_h)
&\leq v(x_h,t_0) -  \int_{t_0}^{t_0+h} L_0(\eta_h, t, \dot \eta_h) \, dt + h \, \om(h)\\
&\leq  v(\zeta(t_0+h),t_0+h) + h \, \om(h).
\end{align*}
On the other hand, by \eqref{teogigionzi2} and \eqref{teogigionzi3},
\begin{align*}
F^\varphi_h(x_0)
&\geq v(x_0,t_0) - h L_0(x_0,t_0,\dot\zeta(t_0)) -h \, \om(h) \\
&\geq v(x_0,t_0)- h \left[ \frac{1}{h} \int_{t_0}^{t_0+h} L_0(\zeta,t,\dot\zeta) \, dt  + \om(h) \right]  -h \, \om(h)\\
&= v(\zeta(t_0+h),t_0+h) - 2 h \,\om(h).
\end{align*}
Combining the previous inequalities, we obtain
\begin{align*}
0
&\leq F^\varphi_h(x_h) -F^\varphi_h(x_0)\\
&\leq v(\zeta(t_0+h),t_0+h) + h \, \om(h) - v(\zeta(t_0+h),t_0+h) + 2 h \, \om(h) \\
&\leq 3 h \, \om(h).
\end{align*}
Dividing by $h$ and expanding $F^\varphi_h$, we find
\begin{align}
0
&\leq \frac{1}{h} \Bigl[\varphi(x_h) - \varphi(x_0) + h \Bigl( L_0\left(x_0,t_0, \frac{\zeta(t_0+h) - x_0}{h} \right) - L_0\left(x_0,t_0, \frac{\zeta(t_0+h) - x_h}{h} \right) \Bigr) \Bigr] \label{teogigiosi3}\\
&= Dv(x_0,t_0) \cdot \frac{x_h-x_0}{h} + \frac{\mathrm{o}(|x_h-x_0|)}{h} \nonumber\\
&\quad\ + \Bigl[ L_0\left(x_0,t_0, \frac{\zeta(t_0+h) - x_0}{h} \right) - L_0\left(x_0,t_0, \frac{\zeta(t_0+h) - x_h}{h} \right) \Bigr] \nonumber\\
&\leq 3 \,  \om(h). \nonumber
\end{align}
We have by Lemma \ref{gigiosi} {\bf (i)}
\[
\lim_{h} \frac{\zeta(t_0+h)-x_h}{h} = H^0_p(x_0,t_0,-Dv(x_0,t_0))
\]
and consequently
\begin{align*}
\lim_{h} \frac{x_h-x_0}{h}
&= - \lim_{h} \frac{\zeta(t_0+h)-x_h}{h} + \lim_{h} \frac{\zeta(t_0+h)-\zeta(t_0)}{h} \\
&= \dot \zeta(t_0) - H^0_p(x_0,t_0,-Dv(x_0,t_0)).
\end{align*}
Passing to the limit in \eqref{teogigiosi3}, we get
\begin{align*}
&L_0(x_0,t_0,\dot\zeta(t_0)) - L_0(x_0,t_0,H^0_p(x_0,t_0,-Dv(x_0,t_0))) \\
&\quad= (-Dv(x_0,t_0)) \cdot (\dot\zeta(t_0)- H^0_p(x_0,t_0,-Dv(x_0,t_0))),
\end{align*}
which implies
\begin{equation}\label{teogigionzi10}
\dot\zeta(t_0)= H^0_p(x_0,t_0,-Dv(x_0,t_0))
\end{equation}
because $-Dv(x_0,t_0) \in \dfrac{\partial}{\partial q} L_0\big(x_0,t_0,H^0_p(x_0,t_0,-Dv(x_0,t_0))\big)$ by Lemma \ref{gigiosi} {\bf (ii)} and $L_0$ is strictly convex in $q$. Since $t_0$ is a differentiability point of $\zeta$ and a Lebesgue point of $L_0$, we get from \eqref{teogigionzi10} and Corollary \ref{pregigionzi}
\[
(x_0,t_0) \in \A_{g_0}
\txt{and}\qquad
W_{g_0}(x_0,t_0)= H^0_p(x_0,t_0,-Dv(x_0,t_0)).
\]
The assertion is then a consequence of Proposition \ref{master}.
\end{proof}

\smallskip

The argument of Proposition \ref{teogigionzi} plus Lemma \ref{gigiosi} actually proves the following more general statement:

\smallskip

\begin{Corollary}
There exists $B \subset \T^N \times [0,T]$ containing a subset of $\mu_\zeta$--full measure, for any $g_0$--optimal curve $\zeta$,
such that $\A_{g_0} \supset B$, and
\[
W_{g_0}(x,t)= H^0_p(x,t,-D\varphi(x))
\]
for any $(x,t) \in B$ where $v(\cdot,t)$ admits a $C^1$ super- or subtangent $\varphi$.
\end{Corollary}

\bigskip

\section{A fixed point result} \label{fissa}

Here we consider a Hamiltonian
\[
H: \T^N \times \bP(\T^N) \times \R^N \to \R
\]
satisfying the structural assumptions {\bf (H1)}--{\bf (H3)}, up to straightforward adaptations.

We are specifically interested in the case where a measure $\xi \in \bP(\G)$ is fixed and $H$ depends on time only through its evaluation curve $\xi(t)$. In other words, the Hamiltonians we focus on in this section have the form
\[
H(x,\xi(t),p) \txt{for $\xi \in \bP(\G)$.}
\]
By Lemma \ref{preservone}, the above Hamiltonians inherit the structural conditions {\bf (H1)}--{\bf (H3)}. Therefore all results established in the first part of the paper for the {\em abstract} Hamiltonian $H_0$ apply to them for any given $\xi$.

We also consider a continuous function
\[
g: \T^N \times \bP(\T^N) \to \R
\]
and the HJ equation
\begin{equation}\label{HJ} \tag{HJ$^*$}
  - u_t + H(x,\xi(t),- Du) = 0 \quad\hbox{in $ \T^N \times (0,T)$.}
\end{equation}
coupled with the final datum $g(\cdot,\xi(T))$ at $t=T$.

Since both $\T^N$ and $\bP(\T^N)$ are compact, the continuity of $g$ implies the existence of a constant $M_g$ and a modulus $\om_g$ with
\begin{equation}\label{FD0}
 \max_{\T^N} |g(x,\mu)| \leq M_g,
 \qquad
 | g(x_1,\mu) -  g(x_2,\mu)| \leq \om_{ g}(|x_2-x_1|)
\end{equation}
for any $x_1, x_2 \in \T^N$, $\mu \in \bP(\T^N)$.

\medskip

Given an {\em initial measure} $\mu_0$, we define the multivalued map $\cT: \bP(\G) \to \bP(\G)$ by
\begin{equation}\label{defTau}
 \cT(\xi) = \{ \hbox{$g( \cdot,\xi(T))$--optimal measures $ \xi^*$ with $\xi^*(0)= \mu_0$}\}.
\end{equation}

\smallskip

The main result of the section is:

\begin{Theorem}\label{mfg}
Under the assumptions {\bf (H1)}--{\bf (H3)} the multivalued map $\cT$ possesses a fixed point, that is, there exists $\xi^* \in \bP(\G)$ satisfying
\[
\xi^* \in \cT(\xi^*).
\]
\end{Theorem}

\smallskip

For the proof we will rely on the following version of the Kakutani–Glicksberg–Fan fixed point Theorem (see e.g. \cite{AB06}):

\begin{Theorem}\label{KGF}
Let $S$ be a nonempty, compact and convex subset of a locally convex Hausdorff space. Let $\Phi$ be a multivalued map on $S$ with nonempty convex values and closed graph. Then the set of fixed points of $\Phi$ is nonempty and compact.
\end{Theorem}

\smallskip

We now introduce some preliminary material. Set
\[
\bP_0(\G)
= \overline{\co}\bigl\{
\xi \in \bP(\G) \;\hbox{with $\xi(0)=\mu_0$, optimal for $g(\cdot,\xi_0(T))$, for some $\xi_0 \in \bP(\G)$}
\bigr\},
\]
where $\overline{\co}$ stands for the closed convex hull.
Note that
\[
\cT(\bP(\G)) \subset  \bP_0(\G).
\]

\smallskip

We will prove the existence of a fixed point for the restriction of $\cT$ to $\bP_0(\G)$. This relies on a crucial compactness property. For its proof we need the following preliminary lemma, where we exploit the topological properties discussed in Section \ref{topino}.

\begin{Lemma}\label{lemuecompa}
The closed convex hull of a compact subset of $\bP(\G)$ is compact.
\end{Lemma}

The property above does not hold in a general Polish space. To prove it in our setting, we must use the specific structure of $\bP(\G)$. A general result states that the closed convex hull of a compact set in a locally convex space is precompact, in the sense that its completion, with respect to the canonical uniformity induced by the seminorms, is compact, see \cite[Chapter II, Section 4.3]{SW99}.

To prove Lemma \ref{lemuecompa}, we accordingly introduce some local convexity in the picture. We exploit the fact that the space $\mathbb M(\G)$ of signed Borel measures with bounded variation on $\G$ is locally convex with the family of seminorms induced by duality from the space $C_b(\G)$ of bounded continuous functions endowed with the strict topology. The corresponding topology on $\mathbb M(\G)$ is the narrow one. See Section \ref{topino} for details, including all relevant definitions and terminology.

\smallskip

\begin{proof}[{\bf Proof of Lemma \ref{lemuecompa}}]
Let $\mathbb A$ be a compact subset of $\bP(\G)$, and let $\mathbb B$ denote its closed convex hull. Since $\mathbb B \subset \mathbb M(\G)$ it is enough to show, in view of \cite[Chapter II, Section 4.3]{SW99}, that $\mathbb B$ is complete with respect to the uniformity induced by the dual seminorms.

In this context the terms {\em compact} and {\em closed} are unambiguous since both refer to the narrow topology, which is related to the family of dual seminorms in $\mathbb M(\G)$, see \eqref{topino1}, and is also equivalent to the topology associated with Wasserstein distance in $\bP(\G)$. Regarding completeness, however, we must specify that it is meant with respect to the uniformity generated by the dual seminorms: completeness depends on the uniformity, not only on the topology.

There are, in fact, two different uniformities at play: that related to the seminorms in $\mathbb M(\G)$, and that induced by $d_W$.

If we can show that any Cauchy net in $\mathbb B$ with respect to the dual seminorms is also Cauchy for $d_W$ then the claim follows immediately from the fact that $\mathbb B \subset \bP(\G)$ is closed, and $\bP(\G)$ is complete with respect to $d_W$ on $\bP(\G)$.

The proof therefore boils down to showing this implication, which is a consequence of  the Kantorovich--Rubinstein Theorem.

Let in fact $K$ be a compact subset of $\G$ for the uniform topology. Define
\[
\Theta_0=\{f \in C_b(\G) \mid f \ \hbox{is $1$--Lipschitz  and $\min_K f=0$}\},
\]
where Lipschitz continuity is understood with respect to the uniform norm on $\G$. We can choose $a >0$ such that
\[
\|f\|_{\psi_0} \leq 1 \txt{for any $f \in \Theta_0$,}
\]
where $\psi_0= a \, \chi_K$, and $\|\cdot\|_{\psi_0}$ is an admissible seminorm for the strict topology (see Section \ref{topino}). By \eqref{topino1}, we therefore get
\begin{equation}\label{lemuecompa1}
 \|\mu\|^*_{\psi_0} \geq  \sup\{ (\mu,f) \mid f \in \Theta_0\} \txt{for any
 $\mu \in \mathbb M(\G)$,}
\end{equation}
where $(\cdot,\cdot)$ denotes the duality pairing between $\mathbb M(\G)$ and $C_b(\G)$. Let $\mu_\al$ be a Cauchy net, with respect to the dual seminorms, contained in $\mathbb B$. According to Definition \ref{topino2}, there exists, for any $\eps >0$, an index $\al_0$ with
\[
\|\mu_\al - \mu_\be \|_{\psi_0} < \eps \txt{for any $\al, \be  \succ \al_0$.}
\]
By \eqref{lemuecompa1} this implies
\[
|\langle \mu_\al, f \rangle - \langle \mu_\be , f \rangle| < \eps \txt {for any $f \in \Theta_0$.}
\]
The inequality remains valid if we add to $f$ a constant, since such a constant cancels out under subtraction of the two integrals. Hence
\[
|\langle \mu_\al, f \rangle - \langle \mu_\be , f \rangle| < \eps \txt {for any $f \in C_b(\G)$ $1$--Lipschitz.}
\]
Applying the Kantorovich--Rubinstein Theorem to the Polish space $\G$ (\cite[Theorem 1.14]{V03}) we deduce
\[
d_W(\mu_\al,\mu_\be) < \eps,
\]
so that $\mu_\al$ is Cauchy with respect to the first Wasserstein distance. This concludes the argument.
\end{proof}

\smallskip

\begin{Proposition}\label{uecompa}
The set $\bP_0(\G)$ is compact in $\bP(\G)$.
\end{Proposition}

\begin{proof}
We first prove that
\begin{equation}\label{uecompa1}
\Bigl\{
\xi \in \bP(\G) \;\hbox{with $\xi(0)=\mu_0$, optimal for $g(\cdot,\xi_0(T))$, for some $\xi_0 \in \bP(\G)$}
\Bigr\}
\end{equation}
is compact. Indeed, the measures in \eqref{uecompa1} are supported on the same set of curves as in Proposition \ref{compattone}, with $M_g$, $\om_g$ from \eqref{FD0} in place of $M$, $\om$, respectively, which is compact in $\G$ by virtue of Proposition \ref{compattonebis}. This shows the claim.

Since $\bP_0(\G)$ is the closed convex hull of the compact set in \eqref{uecompa1}, the assertion follows from Lemma \ref{lemuecompa}.
\end{proof}

\smallskip

We next introduce some perturbations of the Lagrangian $L$ which will be used in the proof of Theorem \ref{mfg} to compensate for the lack of uniform continuity of $L$ on $\T^N \times \bP(\T^N) \times \R^N$.
Choose $\be_0$ such that
\begin{equation}\label{beta0}
 \beta_0 > \max \{ L(x,\mu,0) \mid (x,\mu) \in \T^N \times \bP(\T^N)\},
\end{equation}
and for each $\be >0$ define
\begin{equation}\label{betaccio0}
 L_\be(x,\mu,q) =  L(x,\mu,q) \wedge (\be \, |q| +\be_0) \txt{for any $(x,\mu,q) \in \T^N \times \bP(\T^N) \times \R^N$,}
\end{equation}
where for any $a, b \in \R$, $a \wedge b = \min \{a,b\}$.
Let $L_\be^*: \T^N \times \bP(\T^N) \times \R^N \to \R \cup \{ +\infty\}$ be its conjugate function
\[
L_\be^*(x,\mu,p) = \sup_q \bigl(p \cdot q-  L_\be(x,\mu,q)\bigr),
\]
and let $L^{**}_\be:  \T^N \times \bP(\T^N) \times \R^N \to \R$ be the biconjugate
\[
L_\be^{**}(x,\mu,q) = \sup_p \bigl(p \cdot q -  L^*_\be(x,\mu,p)\bigr).
\]
See Section \ref{bi} for basic facts on the above conjugate and biconjugate functions. We recall that
\begin{equation}\label{hullo}
\hbox{$q \mapsto  L_\be^{**}(x,\mu,q)$ is the convex hull of $q \mapsto L_\be(x,\mu,q)$ for any $(x,\mu)$,}
\end{equation}
see \cite[Chapter E, Theorem 1.3.5]{UrruL01}, i.e., the largest convex function in $q$ dominated by $L_\be$.

\smallskip

\begin{Lemma}\label{lemdemure1}
Given $\eps >0$, $\be >0$, $\be_0$ as in \eqref{beta0}, and a sequence $\xi_n$ in $\bP(\G)$ narrowly converging to a measure $\xi$, there exists $n_\eps$ such that
\begin{equation}\label{lemedure101}
L(x,\xi_n(t),q)  \geq  L_\be(x,\xi(t),q) -\eps
\end{equation}
for all $(x,t,q) \in \T^N \times [0,T] \times \R^N$ and all $n > n_\eps$.
\end{Lemma}

\begin{proof}
By uniform superlinearity of the Lagrangian $L$, we can choose a closed ball $B$ in $\R^N$ with
\begin{equation}\label{lemdure11}
 L(x,\mu,q) > \be \, |q| + \be_0 = L_\be(x,\mu,q) \txt{for $(x,\mu,q) \in \T^N \times \bP(\T^N) \times (\R^N \setminus B)$.}
\end{equation}
By Lemma \ref{lemfranciosa1} and the uniform continuity of $L$ in $\T^N \times \bP(\T^N) \times B$, we can ensure that, for $n$ sufficiently large,
\begin{equation}\label{lemdure12}
 L(x,\xi_n(t),q) \geq L(x,\xi(t),q) - \eps \geq L_\be(x,\xi(t),q) -\eps
\end{equation}
for $(x,t,q) \in \T^N \times [0,T] \times B$. Inequalities \eqref{lemdure11}, \eqref{lemdure12} yield the assertion.
\end{proof}

\smallskip

Next statement is pictorially evident; we provide nevertheless a formal proof.

\smallskip

\begin{Proposition}\label{betino0}
We have
\[
\lim_{\be \to + \infty}  L^{**}_\be(x,\mu,q) =  L(x,\mu,q)
\txt{pointwise in $\T^N \times \bP(\T^N) \times \R^N$.}
\]
\end{Proposition}

\begin{proof}
Fix $\be >0$, $(x_0,\mu_0,q_0) \in \T^N \times \bP(\T^N) \times \R^N$, and
consider an affine subtangent
\[
\phi(q) =   L(x_0,\mu_0,q_0) + p_0 \cdot(q-q_0)
\txt{for some $p_0 \in \partial_q  L(x_0,\mu_0,q_0)$}
\]
to $q \mapsto  L(x_0,\mu_0,q)$ at $q_0$, where $\partial_q  L(x_0,\mu_0,q_0)$ is the subdifferential of $q \mapsto  L(x_0,\mu_0,q)$ at $q_0$. By \eqref{beta0} we have
\begin{equation}\label{betino01}
 \be_0 >  L(x_0,\mu_0,0) \geq L(x_0,\mu_0,q_0) - p_0 \cdot q_0.
\end{equation}
If $L_\be(x_0,\mu_0,\cdot)$ does not dominate $\phi$ on all of $\R^N$, there exists $q_1$ in the closure of
\[
\{q \mid L_\be(x_0,\mu_0,q) < \varphi(q)\}
\]
with
\[
L(x_0,\mu_0,q_0) + p_0 \cdot(q_1-q_0) = \be_0 + \be \, |q_1| = L_\be(x_0,\mu_0,q_1),
\]
and from \eqref{betino01} we deduce
\[
\be \, |q_1| < p_0 \cdot q_1 \leq |p_0| \, |q_1|,
\]
so that
\begin{equation}\label{betino02}
  |p_0| > \be.
\end{equation}
On the other hand, $q \mapsto  L(x,\mu, q)$ is locally Lipschitz continuous, uniformly with respect to $(x,\mu)$, and superlinear, hence there exists a compact set $K_\be \subset \R^N$ with
\[
\partial_q  L(x_0,\mu_0, q) \subset B(0,\be)
\txt{and} \qquad
L(x_0,\mu_0, q) = L_\be(x_0,\mu_0, q) \txt{for any $q \in K_\be$.}
\]
Combining \eqref{betino02}, \eqref{hullo} and the characterization of the convex hull as the pointwise supremum of all the affine functions dominated by the given function, we get
\[
L(x_0,\mu_0,q) =  L_\be(x_0,\mu_0,q) =  L^{**}_\be(x_0,\mu_0,q) \txt{for $q \in K_\be$.}
\]
The family $K_\be$, for $\be$ varying in $\R^+$, is increasing with respect to inclusion and satisfies
\[
\R^N = \bigcup_{\be >0} K_\be.
\]
This implies the claim.
\end{proof}

\medskip

\begin{proof}[{\bf Proof of Theorem \ref{mfg}}]
Taking into account Proposition \ref{uecompa}, the fact that $\bP_0(\G)$ is a subset of $\mathbb M(\G)$ where the narrow topology is induced by a family of seminorms, and the Kakutani–Glicksberg–Fan Theorem \ref{KGF}, the existence of a fixed point in $\bP_0(\G)$ is a consequence of the following statement:
\begin{itemize}
  \item[--] The multivalued map $\cT:  \bP_0(\G) \to \bP_0(\G)$ has nonempty convex values and closed graph.
\end{itemize}
We now prove it. Let $\xi \in \bP_0(\G)$. By Proposition \ref{esisto}, there exist $g(\cdot,\xi(T))$--optimal measures associated to \eqref{HJ-} with $H(x,\xi(t),q)$ in place of $H_0(x,t,q)$, and with initial distribution $\mu_0$. Such measures belong to $\cT(\xi)$, which is then nonempty.
An element $\xi^*$ in $\cT(\xi)$ is characterized, by Proposition \ref{optimum} (ii) and Theorem \ref{kanton}, by the identity
\begin{equation}\label{demure01}
   \int  \int_0^T L(\zeta,\xi(t),\dot \zeta) \, dt \, d \xi^*(\zeta) = \langle \mu_0, v(\cdot,0) \rangle - \langle \xi^*(T), g(\cdot, \xi(T)) \rangle,
\end{equation}
where $v$ is the LO solution to \eqref{HJ} with final datum $g(x,\xi(T))$. A direct computation shows that if \eqref{demure01} holds for measures $\xi^*_1$, $\xi^*_2$, then it also holds for any convex combination of them. Thus $\cT(\xi)$ is convex.

It remains to show that the graph of $\cT$ is closed. Consider a sequence $(\xi_n, \xi_n^*)$ in the graph of $\cT$ with
\[
 (\xi_n,\xi_n^*)  \to (\xi,\xi^*).
\]
We must prove that $\xi^*  \in \cT(\xi)$.

Let $v_n$, $v$ be the LO solutions to \eqref{HJ} with Hamiltonians $H(x,\xi_n(t),p)$, $H(x,\xi(t),p)$ and final data $g(x,\xi_n(T))$, $g(x,\xi(T))$, respectively. By Lemma \ref{lemfranciosa1} we have $\xi_n(t) \to \xi(t)$ and $g(x,\xi_n(T)) \to g(x,\xi(T))$ uniformly in $[0,T]$ and $\T^N$, respectively. In addition, by Proposition \ref{stabia} $v_n(\cdot,0) \to v(\cdot,0)$ uniformly in $\T^N$. Using optimality of $\xi^*_n$ and the fact that $\xi^*_n(0)= \mu_0$ for any $n$, we deduce from Proposition \ref{optimum} (ii) that
\begin{equation}\label{demure2}
   \int \, \int_0^T L(\zeta, \xi_n(t), \dot \zeta) \, dt \, d \xi^*_n(\zeta) = \langle \mu_0,v_n(\cdot,0)\rangle - \langle \xi^*_n(T),  g(\cdot,\xi_n(T))\rangle,
\end{equation}
for any $n$. By Lemma \ref{lemdemure3} and the convergences above,
\[
\langle \mu_0,v_n(\cdot,0)\rangle \to \langle \mu_0,v(\cdot,0)\rangle,
\qquad
\langle \xi^*_n(T),  g(\cdot,\xi_n(T))\rangle \to \langle \xi^*(T),  g(\cdot,\xi(T))\rangle.
\]
Passing to the limit in \eqref{demure2}, we obtain
\begin{equation}\label{demure11}
  \lim_n  \int  \int_0^T  L(\zeta, \xi_n(t), \dot \zeta) \, dt \, d\xi^*_n(\zeta) =  \langle \mu_0,v(\cdot,0)\rangle  - \langle \xi^*(T), g(\cdot,\xi(T))\rangle.
\end{equation}
Fix $\be >0$, $\eps >0$. Since $L_\be^{**}$ is convex in $q$, and continuous in all arguments (Proposition \ref{graziolla}), the functional
\[
\zeta \mapsto \int_0^T  L_\be^{**}(\zeta,\xi(t),\dot \zeta) \, dt
\]
is lsc in $\G$ by \cite[Theorem 3.6]{BuGH98} for any measure $\xi$.
Using Proposition \ref{paglia} and Lemma \ref{lemdemure1}, we get
\begin{align*}
\lim_n \int \int_0^T  L(\zeta,\xi_n(t),\dot\zeta) \, dt  \, d\xi^*_n
&\geq \liminf_n   \int  \int_0^T \bigl[   L_\be(\zeta,\xi(t),\dot\zeta) - \eps \bigr] \, dt  \, d\xi^*_n \\
&\geq \liminf_n \int \int_0^T \bigl[  L^{**}_\be (\zeta,\xi(t),\dot\zeta) - \eps \bigr] \, dt \,d\xi^*_n \\
&\geq \int \int_0^T  \bigl[ L^{**}_\be(\zeta,\xi(t),\dot\zeta) - \eps \bigr]\, dt  \, d\xi^*.
\end{align*}
Letting first $\eps \to 0$ and then $\be \to +\infty$, and using Proposition \ref{betino0} together with the monotone convergence theorem, we get
\[
\int_0^T  L^{**}_\be(\zeta,\xi(t),\dot\zeta)\, dt \to  \int_0^T L(\zeta,\xi(t),\dot\zeta) \, dt  \txt {for any $\zeta \in \G$,}
\]
which implies, again using the monotone convergence theorem,
\[
\int \int_0^T  L^{**}_\be(\zeta,\xi(t),\dot\zeta)\, dt \, d \xi^* \to \int \int_0^T L(\zeta,\xi(t),\dot\zeta) \, dt \, d\xi^*.
\]
We finally derive
\begin{equation}\label{demure1}
   \lim_n \int \int_0^T L(\zeta,\xi_n(t),\dot\zeta) \, dt  \, d\xi^*_n  \geq \int \int_0^T    L(\zeta,\xi(t),\dot\zeta) \, dt  \, d\xi^*.
\end{equation}
By combining \eqref{demure11}, \eqref{demure1} and exploiting Proposition \ref{narduz}, we get
\begin{align*}
 \int\int_0^T L(\zeta,\xi(t),\dot\zeta) \, dt  \, d\xi^*
&\leq  \langle \mu_0,v(\cdot,0)\rangle - \langle \xi^*(T), g(\cdot,\xi(T))\rangle \\
&\leq S_T(\mu_0,\xi^*(T)).
\end{align*}
By Proposition \ref{pizza} equality must prevail in the above formula. This yields that $\xi^* \in \cT(\xi)$ by Proposition \ref{optimum} (ii), and concludes the argument.
\end{proof}

\bigskip

\section{MFG systems in a generalized form}\label{generale}

In this section we assume, in addition to {\bf (H1)}--{\bf (H3)}, the additional condition {\bf (H4)} on $H(x,\mu,p)$. By the results of Section \ref{field}, these hypotheses guarantee that for any $\xi \in \bP(\G)$ and any continuous function
\[
g: \T^N \times  \bP(\T^N) \to \R
\]
there exists a Borel vector field
\[
W_{g(\cdot,\xi(T))}: \T^N  \times [0,T] \to \R^N
\]
enjoying all the properties established in Sections \ref{field} and \ref{continual}.

\smallskip

Given $\mu_0 \in \bP(\T^N)$, we consider the system
\begin{equation}\label{MFG0}  \tag{MFG}
  \left \{ \begin{array}{ll}
  -u_t +  H(x, \xi(t),- Du) =0   &  \txt{in $\T^N \times [0,T)$},\\[2pt]
  u(x,T)=  g(x, \xi(T))   & \txt{for $x \in \T^N$}, \\[2pt]
  \xi_t(t) + \nabla \cdot \bigl(\xi(t) W_{g(\cdot,\xi(T))}(x,t) \bigr) =0 & \txt{in $\T^N \times [0,T]$},\\[2pt]
  \xi(0)= \mu_0  &
           \end{array} \right .
\end{equation}

This is not the classical MFG system, since the vector field driving the continuity equation is
\[
W_{g(\cdot,\xi(T))}(x,t)
\]
instead of the usual $H_p(x,t,-Dv(x,t))$.

\smallskip

\begin{Definition}
A pair $(v,\xi^*)  \in C(\T^N \times [0,T]) \times \bP(\G)$
is said to be a {\em solution} of \eqref{MFG0} if
\begin{itemize}
  \item[--]  $v$ is the LO solution of the HJ equation with final datum $g(\cdot,\xi^*(T))$;
  \item[--] the evaluation curve of $\xi^*$ is a solution to the continuity equation with initial datum $\mu_0$.
\end{itemize}
\end{Definition}

\smallskip

Recall the multivalued map $\cT: \bP(\G) \to \bP(\G)$ defined by
\[
\cT(\xi) = \{ \hbox{$g( \cdot,\xi(T))$--optimal measures $\xi_0$ with $\xi_0(0)= \mu_0$}\}.
\]

\smallskip

\begin{Theorem}\label{teoremone}
Assume {\bf (H1)}--{\bf (H4)}. For every $\mu_0 \in \bP(\T^N)$ and $g \in C(\T^N \times \bP(\T^N))$, any fixed point $\xi^*$ of $\cT$, together with the LO solution $v$ of the corresponding HJ equation with final datum $g(\cdot,\xi^*(T))$, provides a solution of \eqref{MFG0}. If, in addition, {\bf (H5)} holds then the curve $\xi^*(t)$ is $(1+\eps_0)^*$--absolutely continuous.
\end{Theorem}

\begin{proof}
By Theorem \ref{mfg}, there exists a fixed point $\xi^* \in \bP(\G)$ of $\cT$. By Theorem \ref{assone} $\xi^*(t)$ then solves the continuity equation
\[
\xi^*_t(t) + \nabla \cdot \bigl(\xi^*(t) \, W_{g(\cdot,\xi^*(T))}(x,t) \bigr) =0.
\]
This yields the first part of the assertion.
If {\bf (H5)} holds the absolute continuity of $\xi^*(t)$ with exponent
$(1+\eps_0)^*$ follows from Theorem \ref{assone}.
\end{proof}

\bigskip

\appendix

\section{Disintegration theory}\label{disintegra0}

We briefly recall some standard facts on the disintegration of measures (see, e.g., \cite{DM78}).

\smallskip

\begin{Definition}\label{misuramisura}
Let $X$, $Y$ be Polish spaces. A family $\eta_y$, with $y$ varying in $Y$, of probability measures on $X$ is said {\em measurable} if
\begin{equation}\label{misuramisura1}
 y \mapsto \eta_y(B)  \quad \hbox{is Borel for any Borel subset $B$ of $X$}.
\end{equation}
\end{Definition}

Given such a family, and a Borel probability measure $\nu$ on $Y$, we define a probability measure $\mu$ on $X$ by the {\em pull--back} formula
\begin{equation}\label{pullback}
  \mu(B)= \int \eta_y(B) \, d \nu.
\end{equation}

\smallskip

We next consider a Borel map $\Phi: X \to Y$ satisfying
\begin{equation}\label{A}
  \mathrm{spt}(\eta_y) \subset \Phi^{-1}(y) \qquad\hbox{for $\nu$--almost all  $y \in Y$}.
\end{equation}

\begin{Proposition}\label{dis}
Assume \eqref{A} holds and define $\mu$ by \eqref{pullback}. Then $\mu$ is a probability measure on $X$ and
\[
  \nu = \Phi \# \mu.
\]
\end{Proposition}

\begin{proof}
Since \eqref{pullback} defines a probability measure, it remains to check that
\[
  \mu(\Phi^{-1}(E)) = \nu(E)
\]
for any Borel set $E$ in $Y$. Using \eqref{A}, we compute:
\begin{eqnarray*}
   \mu(\Phi^{-1}(E)) &=& \int \eta_y(\Phi^{-1}(E)) \, d \nu
    = \int \eta_y(\Phi^{-1}(E) \cap \Phi^{-1}(y)) \, d \nu \\
  &=&  \int_E \eta_y(\Phi^{-1}(E) \cap \Phi^{-1}(y)) \, d \nu
   = \int_E \eta_y( \Phi^{-1}(y)) \, d \nu \\
  &=& \int_E \eta_y(X) \, d\nu = \nu(E).
\end{eqnarray*}
\end{proof}

\smallskip

The disintegration theorem allows a more precise construction. Given a probability measure $\mu$ on $X$ and defining $\nu = \Phi \# \mu$, a disintegration of $\mu$ with respect to $\Phi$ consists in a measurable family of probability measures $\eta_y$ with $y\in Y$ supported on $\Phi^{-1}(y)$ for $\nu$--almost all $y$, such that
\begin{equation}\label{disintegra}
   \int   g  \, d\mu= \int \left [ \int g \, d \eta_y \right ] \, d \nu
\end{equation}
for any $\mu$--integrable function $g$.

\smallskip
The disintegration theorem reads:

\smallskip

\begin{Theorem}\label{disbis}
Let $X$ and $Y$ be Polish spaces and $\Phi: X \to Y$ a Borel map. Then any Borel probability measure $\mu$ on $X$ admits a disintegration with respect to $\Phi$.
\end{Theorem}

\bigskip

\section{The LO solution} \label{LxO}

The aim of this appendix is to prove the following.

\begin{Theorem}\label{LxxO}
Let $g_0$ be a continuous function  on $\T^N$. Under assumptions {\bf (H1)}--{\bf (H3)} the function $v$ given by the Lax--Oleinik formula \eqref{LO} is a continuous solution of \eqref{HJ-} agreeing with $g_0$ at $t=T$.
\end{Theorem}

We first collect some auxiliary results.

\smallskip

Set
\begin{eqnarray*}
  M_{g_0} &=& \max_{\T^N} g_0,\\
  m_{g_0} &=& \min_{\T^N} g_0.
\end{eqnarray*}

\begin{Lemma}\label{lemiran}
The function $v$ given by the Lax--Oleinik formula is bounded on $\T^N \times [0,T]$.
\end{Lemma}

\begin{proof}
We have
\begin{eqnarray*}
  v(x_0,t_0) &\leq& g_0(x_0) + \int_{t_0}^T L_0(x_0,t,0) \, dt \leq M_0 \, T + M_{g_0},\\
  v(x_0,t_0) &\geq& m_{g_0} + m_{L_0} \, T
\end{eqnarray*}
for any $(x_0,t_0) \in \T^N \times [0,T]$, where $M_0$, $m_{L_0}$ are the constants appearing in Notation \ref{M0mL}.
\end{proof}

\smallskip

\begin{Lemma}\label{preiran}
Let $x_0 \in \T^N$ and $t_1 < t_2$. Then
\[
  v(x_0,t_1) \leq v(x_0,t_2) + M_0 \, (t_2 - t_1).
\]
\end{Lemma}

\begin{proof}
Let $\zeta:[t_2,T] \to \T^N$ be a curve with
\[
  v(x_0,t_2) = g_0(\zeta(T)) + \int_{t_2}^T L_0(\zeta,t,\dot\zeta) \, dt.
\]
Define $\overline \zeta$ by
\[
 \overline \zeta(t)=
 \begin{cases}
   x_0 & \text{in $[t_1,t_2]$},\\
   \zeta(t) & \text{in $[t_2,T]$}.
 \end{cases}
\]
Using the Lax--Oleinik formula with $\overline \zeta$ as competitor, we get
\[
  v(x_0,t_1) \leq  g_0(\overline \zeta(T)) + \int_{t_1}^T  L_0(\overline \zeta,t,\dot{\overline \zeta}) \, dt
  \leq v(x_0,t_2) + M_0 \, (t_2 - t_1).
\]
\end{proof}

\smallskip

\begin{Lemma}\label{prepostiran}
For every $t_0 \in [0,T]$, the function $x \mapsto v(x,t_0)$ is lsc on $\T^N$.
\end{Lemma}

\begin{proof}
For $t_0=T$ the claim is trivial since $v(\cdot,T)=g_0$ is continuous. Assume $t_0 < T$ and fix $x_0 \in \T^N$. Let $x_n$ be a sequence converging to $x_0$. For each $n$, let $\zeta_n$ be a curve with
\[
  v(x_n,t_0) =  g_0(\zeta_n(T)) +  \int_{t_0}^T L_0(\zeta_n, t, \dot \zeta_n) \, dt.
\]
By Lemma \ref{lemiran}, the integrals in the right–hand side of the above formula are equibounded, hence by Proposition \ref{sciare} the curves $\zeta_n$ uniformly converge in $[t_0,T]$, up to a subsequence, to a curve $\zeta$ with $\zeta(t_0)= x_0$. By Proposition \ref{presciare}
\[
  \liminf_n \bigg( g_0(\zeta_n(T)) + \int_{t_0}^T L_0(\zeta_n, t, \dot  \zeta_n) \, dt \bigg)
  \geq  g_0(\zeta(T)) + \int_{t_0}^T L_0(\zeta, t, \dot  \zeta) \, dt
  \geq v(x_0,t_0),
\]
which eventually yields
\[
  \liminf_n v(x_n,t_0) \geq v(x_0,t_0).
\]
\end{proof}

\smallskip

\begin{Proposition}\label{iran}
The function $v$ given by the Lax--Oleinik formula is lsc in $\T^N \times [0,T]$.
\end{Proposition}

\begin{proof}
Fix $(x_0,t_0) \in \T^N \times [0,T)$. Let $(x_n,t_n)$ be a sequence converging to $(x_0,t_0)$. Fix $\eps$ with $t_0+ \eps < T$, and take $n$ large enough so that $t_n < t_0 +\eps$. Denote by $\zeta_n: [t_n,T] \to \T^N$ a sequence of curves with
\begin{equation}\label{iran01}
 v(x_n,t_n) =  g_0(\zeta_n(T)) +  \int_{t_n}^T L_0(\zeta_n, t, \dot \zeta_n) \, dt.
\end{equation}
Set
\[
  y_n = \zeta_n(t_0 + \eps)
\]
so that
\[
  v(y_n,t_0 + \eps) =  g_0(\zeta_n(T)) + \int_{t_0 + \eps}^T L_0(\zeta_n,t,\dot\zeta_n) \, dt.
\]
Thus
\begin{equation}\label{iran3}
  v(x_n,t_n) =  v(y_n,t_0+\eps) + \int_{t_n}^{t_0+\eps} L_0(\zeta_n,t,\dot\zeta_n) \, dt
  \geq v(y_n,t_0+\eps) +  m_{L_0} \, (t_0 + \eps -t_n).
\end{equation}
Extend each $\zeta_n$ to $[0,T]$ by setting
\[
  \zeta_n(t) = \zeta_n(t_n) =x_n \quad \text{in $[0,t_n)$}.
\]
By Lemma \ref{lemiran} and Proposition \ref{sciare}, the extended curves converge uniformly in $[0,T]$, up to subsequences, to a curve $\zeta$. In particular
\begin{equation}\label{iran4}
  x_0 = \lim_n x_n =  \lim_n \zeta_n(t_n)= \zeta(t_0)
  \quad\hbox{and}\quad
  \lim_n y_n = \lim_n \zeta_n(t_0+\eps)=  \zeta(t_0 + \eps).
\end{equation}
From \eqref{iran3}, \eqref{iran4} and Lemmata \ref{preiran}, \ref{prepostiran} we get
\begin{eqnarray}
 \liminf_n v(x_n,t_n) &\geq& \liminf_n v(y_n,t_0 +\eps) + m_{L_0} (t_0+\eps -t_n) \nonumber \\
   &\geq& v(\zeta(t_0+\eps),t_0+\eps) + m_{L_0} \, \eps
   \geq  v(\zeta(t_0+\eps),t_0) + (m_{L_0}- M_0) \eps.
    \label{iran02}
\end{eqnarray}
Letting $\eps \to 0$ and using again Lemma \ref{prepostiran} and \eqref{iran4} we deduce
\[
  \liminf_n v(x_n,t_n)
  \geq \lim_{\eps \to 0} \big( v(\zeta(t_0+\eps),t_0) + (m_{L_0}-M_0) \eps \big)
  \geq   v(x_0,t_0),
\]
which shows the claimed lower semicontinuity in $\T^N \times [0,T)$.

We now consider the case of a point $(x_0,T)$. Let $(x_n,t_n)$ be a sequence in $\T^N \times [0,T)$ converging to $(x_0,T)$, and let $\zeta_n$ be a sequence of curves satisfying \eqref{iran01}. Then
\begin{equation}\label{iran02T}
  v(x_n,t_n) \geq g_0(\zeta_n(T)) +  m_{L_0} \, (T-t_n).
\end{equation}
Extending the curves $\zeta_n$ as before, they are equicontinuous and uniformly convergent, up to a subsequence, to a curve $\zeta$. Let $\omega$ be a common modulus of continuity, then
\[
  |x_0 - \zeta_n(T)| \leq |x_0 - x_n| +  |\zeta_n(t_n) - \zeta_n(T)| \leq |x_0 - x_n| + \omega(T-t_n),
\]
hence
\[
  \lim_n \zeta_n(T) = x_0.
\]
From \eqref{iran02T}, we obtain
\[
  \liminf_n v(x_n,t_n) \geq \lim_n \big( g_0(\zeta_n(T)) +  m_{L_0} \, (T-t_n) \big) = g_0(x_0).
\]
This concludes the argument.
\end{proof}

\smallskip

\begin{Proposition}\label{iranbis}
The function $v$ given by the Lax--Oleinik formula is usc on $\T^N \times [0,T]$.
\end{Proposition}

\begin{proof}
Fix $(x_0,t_0) \in \T^N \times [0,T]$ and let  $(x_n,t_n)$ be a  sequence  converging to $(x_0,t_0)$ . Let $\eps >0$ be fixed. Let  $\zeta:[t_0,T] \to \T^N$ be a curve with
\[v(x_0,t_0) = g_0(\zeta(T)) + \int_{t_0}^T L_0(\zeta,t,\dot\zeta) \, dt.\]
For $n$ with $t_0 > t_n-\eps$  and $ \frac{|x_0 -x_n|}{t_0-t_n+\eps} < 1$, we define
\[  \zeta_n(t)= \left \{ \begin{array}{cc}
                        \eta_n(t) & \quad\hbox{in $[t_n-\eps ,t_0]$} \\
                         \zeta(t) &  \hbox{in $[t_0,T]$ }
                     \end{array} \right . ,\]
 where $\eta_n :[t_n-\eps , t_0] \to \T^N$ is the geodesic joining  $x_n$ to $x_0$. In particular
 \[|\dot \eta_n (t)|= \frac{|x_0 -x_n|}{t_0-t_n+\eps} < 1.\]
By definition of $v$
\begin{eqnarray*}
  v(x_n,t_n-\eps) &\leq&  g_0(\zeta_n(T)) + \int_{t_n-\eps}^T L_0(\zeta_n,t,\dot\zeta_n) \, dt \leq
\int_{t_n-\eps}^{t_0} L_0(\eta_n,t, \dot\eta_n  ) \, dt + v(x_0,t_0) \\
   &\leq&  M_1 \, (t_0-t_n + \eps) + v(x_0,t_0),
\end{eqnarray*}
where
\[M_1 = \max \{|L_0(x,t,q) \mid |q| \leq  1\}.\]
Fix $n$ and let $\eps \to 0$. Using the lower semicontinuity of $v$ proved in Proposition \ref{iran}, we get from the above inequality
\[ v(x_n,t_n) \leq \liminf_{\eps \to 0} v(x_n,t_n-\eps) \leq v(x_0,t_0) + M_1 \, (t_0-t_n) \txt{for $n$ large enough.}\]
Letting $n \to +\infty$, this yields
\[\limsup_n v(x_n,t_n) \leq v(x_0,t_0),\]
hence the claimed upper semicontinuity.
\end{proof}

\smallskip

\begin{proof}[{\it Proof of Theorem \ref{LxxO}}]
By Propositions \ref{iran} and \ref{iranbis}, $v$ is continuous in $\T^N \times [0,T]$, and, by its very definition, it coincides with $g_0$ at $t=T$.

We now prove that it is a viscosity supersolution of \eqref{HJ-}. Let $\psi$ be a $C^1$ strict subtangent to $v$ at some $(x_0,t_0) \in \T^N \times (0,T)$. Let $\zeta: [t_0,T] \to \T^N$ be an optimal curve for $v(x_0,t_0)$, and choose a sequence $t_n \in [t_0,T]$ converging to $t_0$. By the subtangency of $\psi$,
\begin{eqnarray*}
   && \int_{t_0}^{t_n} \big( \psi_t(\zeta(t),t) + D\psi(\zeta(t)) \cdot \dot \zeta(t) \big) \, dt  \\
 &=& \psi(\zeta(t_n),t_n) - \psi(x_0,t_0) \leq v(\zeta(t_n),t_n) - v(x_0,t_0)\\
 &=&  \int_{t_0}^{t_n} -L_0(\zeta(t),t,\dot\zeta(t)) \, dt.
\end{eqnarray*}
Therefore, there exists a sequence of differentiability times $s_n \in [t_0,t_n]$ for $\zeta$ with
\begin{equation}\label{linda1}
  - \psi_t(\zeta(s_n),s_n) - D\psi(\zeta(s_n)) \cdot\dot \zeta(s_n) \geq
L_0(\zeta(s_n),s_n,\dot\zeta(s_n))
\end{equation}
for any $n$. Let $U$ be a closed neighborhood of $(x_0,t_0)$ containing the points $(\zeta(s_n),s_n)$ for $n$ large. Set
\[
  M = \max_U |D\psi(x,t)|.
\]
From \eqref{linda1},
\[
  M |\dot\zeta(s_n)| \geq - D\psi(\zeta(s_n)) \cdot  \dot \zeta(s_n)
  \geq  L_0(\zeta(s_n),s_n,\dot\zeta(s_n)) + \psi_t(\zeta(s_n),s_n).
\]
Since $L_0$ is superlinear in $q$, the sequence $\dot\zeta(s_n)$ must be bounded, hence converges, up to a subsequence, to some $q \in \R^N$. Letting $n \to + \infty$ in \eqref{linda1},
\[
  - \psi_t(x_0,t_0) - D\psi(x_0,t_0) \cdot q \geq L_0(x_0,t_0,q),
\]
and consequently
\begin{equation}\label{linda01}
  - \psi_t(x_0,t_0) + H(x_0,t_0, -D\psi(x_0,t_0)) \geq 0,
\end{equation}
which shows that $v$ is a supersolution in $\T^N \times (0,T)$. The supersolution property at $t=0$ follows by standard extension arguments for time–dependent Hamilton–Jacobi equations.

Next we prove that $v$ is a subsolution to \eqref{HJ-}. Let $\varphi$ be a $C^1$ strict supertangent to $v$ at a point $(x_0,t_0) \in \T^N \times [0,T)$. Then, for all $q \in \R^N$ and all $h >0$ sufficiently small,
\[
\frac{\varphi(x_0,t_0) - \varphi(x_0 + h q, t_0+h)}{h}
 \leq \frac{ v(x_0,t_0)- v(x_0 + h q, t_0+h)}{h}.
\]
Since the argument is local, we may regard $\T^N$ as $\R^N$ in a chart around $x_0$. We get
\[
  \frac{\varphi(x_0,t_0) - \varphi(x_0 + h q, t_0+h)}{h}
  \leq \frac 1{h} \int_{t_0}^{t_0+h} L_0( x_0 + (t-t_0) q, t, q) \, dt.
\]
Letting $h \to 0$, we get
\[
  -\varphi_t(x_0,t_0) - D\varphi(x_0,t_0) \cdot q \leq L_0(x_0, t_0,q).
\]
Since $q \in \R^N$ is arbitrary, this inequality implies
\[
  -\varphi_t(x_0,t_0)+ H(x_0,t_0, -D \varphi(x_0,t_0)) \leq 0.
\]
This ends the proof.
\end{proof}

\smallskip

We close the section with a stability result.

\smallskip

\begin{Proposition}\label{stabia}
Let $g_n$ be a sequence of continuous functions on $\T^N$ converging uniformly to a function $g_0$. The LO solutions $v_n$ of \eqref{HJ-} with final datum $g_n$ converge uniformly in $\T^N \times [0,T]$ to the LO solution $v$ of \eqref{HJ-} with final datum $g_0$.
\end{Proposition}

\begin{proof}
Fix $(x_0,t_0)$ in $\T^N \times [0,T)$, an approximating sequence $(x_n,t_n)$ and $\eps >0$. Assume
\[
  v_n(x_n,t_n) = g_n(\zeta_n(T)) + \int_{t_n}^T L_0(\zeta_n,t,\dot\zeta_n) \, dt
\]
for some curves $\zeta_n$ with $\zeta_n(t_n)= x_n$. For $n$ large, we have $t_n \geq t_0 - \eps$. Extend each $\zeta_n$ to $[t_0-\eps,T]$ by setting
\begin{equation}\label{stabia01}
  \zeta_n(t) = x_n  \quad \text{for $t \in [t_0 -\eps,t_n]$.}
\end{equation}
Then
\[
  v_n(x_n,t_n)
  = g_n(\zeta_n(T)) + \int_{t_0-\eps}^T L_0(\zeta_n,t,\dot\zeta_n) \, dt
  - \int_{t_0 - \eps}^{t_n} L_0(x_n,t,0) \, dt.
\]
By the definition of LO solution and \eqref{stabia01}, the actions of the curves $\zeta_n$ are bounded in $[t_0- \eps,T]$, so by Proposition \ref{sciare} they converge uniformly, up to subsequences, to a curve $\zeta$ on $[t_0-\eps,T]$. Moreover, since the action is lsc with respect to the uniform convergence by Proposition \ref{presciare} and
\[
  \int_{t_0-\eps}^{t_n} L_0(x_n,t,0) \, dt \leq  M_0 \, (t_n-t_0+ \eps),
\]
we get
\[
  \liminf_n v_n(x_n,t_n)
  \geq g_0(\zeta(T)) + \int_{t_0 - \eps}^T L_0(\zeta,t,\dot\zeta) \, dt -  M_0 \, \eps
  \geq v(x_0,t_0 -\eps) - M_0 \, (t_n-t_0+ \eps).
\]
Since $\eps$ is arbitrary and $v$ is continuous,
\begin{equation}\label{stabia1}
 \liminf_n v_n(x_n,t_n) \geq v(x_0,t_0).
\end{equation}

Now let $n$ be so large that
\[
  \|g_n-g_0\|_\infty < \eps.
\]
For such $n$ and a suitable choice of curves $\eta_n:[t_n,T] \to \T^N$ with $\eta_n(t_n)=x_n$ we have
\begin{eqnarray*}
  v(x_n,t_n) &=& g_0(\eta_n(T)) + \int_{t_n}^T L_0(\eta_n,t,\dot\eta_n) \, dt \\
   &\geq& g_n(\eta_n(T)) - \eps + \int_{t_n}^T L_0(\eta_n,t,\dot\eta_n) \, dt \\
   &\geq& v_n(x_n,t_n)- \eps.
\end{eqnarray*}
Passing to the limit as $n$ goes to infinity, we get
\[
  v(x_0,t_0) = \lim_n v(x_n,t_n) \geq \limsup_n v_n(x_n,t_n)- \eps.
\]
The constant $\eps$ being arbitrary, we derive
\begin{equation}\label{stabia2}
  \limsup_n v_n(x_n,t_n) \leq v(x_0,t_0).
\end{equation}
Inequalities \eqref{stabia1}, \eqref{stabia2} imply the assertion.
\end{proof}

\bigskip

\section{Strict topology}\label{topino}

We recall the notion of {\em strict topology} on $C_b(\G)$ (see \cite{HJ72}), the space of bounded continuous functions on $\G$. It is defined through the family of seminorms
\[
  \{ \|\cdot \|_\psi  \mid \psi \in B_0(\G)\},
\]
where $B_0(\G)$ is the space of bounded Borel functions on $\G$ {\em vanishing at infinity} and
\[
  \|f\|_\psi= \|\psi \, f\|_\infty   \quad \text{for any $f \in C_b(\G)$.}
\]
A function $\psi$ is said to vanish at infinity if for every $\eps >0$ there exists a compact subset $K$ of $\G$ with
\[
  |\psi(x)| < \eps \quad \text{in $\G \setminus K$.}
\]
A basis of neighborhoods of a fixed $f_0$ is then given by the sets
\[
  \{ f \in C_b(\G) \mid \|f - f_0\|_{\psi_i} < \eps\} \qquad
   \eps > 0, \, M \in \N, \, \psi_i \in B_0(\G),\,  i=1,\dots,M.
\]
The space $C_b(\G)$ endowed with the strict topology is therefore locally convex, Hausdorff and, moreover, complete with respect to the uniformity induced by these seminorms, see \cite[Theorem 1]{HJ72}.

\smallskip

\begin{Remark}
The classical definition of strict topology (see \cite{Bu58}, \cite{Co67}) is formulated for bounded functions defined on locally compact spaces and involves seminorms $\|\cdot\|_\psi$ with $\psi$ continuous and vanishing at infinity. In our setting $\G$ is an infinite dimensional Banach space, and the only continuous function vanishing at infinity is the zero function. Therefore, to extend the theory to non–locally compact spaces, the continuity assumption on $\psi$ must be removed.
\end{Remark}

The key feature of the strict topology on $C_b(\G)$ is that its topological dual is precisely the space of signed Borel measures with bounded variation on $\G$, denoted by $\mathbb M(\G)$; see \cite[Theorem 2]{HJ72}. In particular, the normalized positive elements of $\mathbb M(\G)$ are exactly the probability measures $\bP(\G)$. We denote by $( \cdot, \cdot )$ the corresponding duality pairing. Note that
\[
  (\mu,f) = \langle \mu, f\rangle \quad \text{for any $f \in C_b(\G)$, $\mu \in \bP(\G)$,}
\]
see Notation \ref{postcedric}.

The weak--star topology on the dual space $\mathbb M(\G)$, i.e., pointwise convergence on $C_b(\G)$, coincides with the narrow topology, and is induced by the family of dual seminorms
\begin{equation}\label{topino1}
  \|\mu\|_\psi ^* = \sup \{ |\langle \mu,f \rangle| \mid f \in C_b(\G), \, \|f\|_\psi \leq 1 \}
\end{equation}
for $\psi$ varying in $B_0(\G)$.
Thus $\mathbb M(\G)$ is locally convex and in addition Hausdorff, see \cite[Chapter 28.15]{Sc97}.

Narrow convergence of a net $\mu_\alpha \to \mu$ can be equivalently described by
\[
  \|\mu_\alpha - \mu\|_\psi^* \to  0  \quad \text{for any $\psi \in B_0(\G)$.}
\]

\smallskip

Unsurprisingly, we say that a subset $\mathbb A$ of $\mathbb M(\G)$ is {\em complete} if every Cauchy net contained in $\mathbb A$ converges to an element of $\mathbb A$.

\smallskip

\begin{Definition}\label{topino2}
A net $\mu_\al$ contained in $\mathbb M(\G)$ is said {\em Cauchy} (with respect to the seminorms $\|\cdot\|^*_\psi$) if for every $\eps >0$ and $\psi \in B_0(\G)$ there exists $\al_0$ such that
\[
  \al, \, \be \succ \al_0 \Rightarrow \|\mu_\al - \mu_\be \|_\psi^* < \eps.
\]
\end{Definition}

\bigskip

\section{Conjugate and biconjugate of the perturbed Lagrangian $ L_\be$}\label{bi}

We consider the Lagrangian
\[
  L: \T^N \times \bP(\T^N) \times \R^N \to \R
\]
introduced in Section \ref{fissa}, and its perturbation $ L_\be$, defined in \eqref{betaccio0}. For convenience, we recall that
\[
  L_\be (x,\mu,q)= L(x, \mu,q) \wedge (\be \, |q| + \be_0),
\]
with $\be >0$, and $\be_0$ satisfying
\begin{equation}\label{beta00}
 \beta_0 > \max \{ L(x,\mu,0) \mid (x,\mu) \in \T^N \times \bP(\T^N)\}.
\end{equation}
We further denote by $\ov{B(0,\be)}$ the closed ball of radius $\be$ centered at $0$ in $\R^N$.

The conjugate of $ L_\be$ is given by
\[
  (x,\mu,p) \mapsto  L^*_\be(x,\mu,p) = \sup_{q \in \R^N} \big( p \cdot q -  L_\be(x,\mu,q)\big),
\]
and the biconjugate $ L^{**}_\be$ in the usual way.

Our aim is to prove:

\begin{Proposition}\label{graziolla}
The function $(x,\mu,q) \mapsto L^{**}_\be(x,\mu,q)$ is continuous.
\end{Proposition}

\smallskip

We begin with two preliminary lemmata.

\begin{Lemma}\label{grazia}
We have $ L_\be^*(x,\mu,p) = + \infty$ if and only if $|p| >\be$.
\end{Lemma}

\begin{proof}
Take $p$ with $|p| > \be$. Then
\[
  L_\be^*(x,\mu,p) \geq \lim_{\la \to + \infty}
  \big( \la |p| -  L_\be(x,\mu, \la p/|p|) \big)
  = \lim_{\la \to + \infty} \big( \la |p| - \be \, \la  - \be_0\big)= + \infty.
\]
Conversely, assume $L_\be^*(x,\mu,p) = + \infty$. Then there exists a sequence $q_n$ with $|q_n| \to + \infty$ such that
\[
  + \infty =  \lim_n \big( p \cdot q_n -  L_\be(x,\mu,q_n)\big)
  = \lim_n \big( p \cdot q_n - \be \, |q_n| - \be_0\big)
  \leq  \lim_n \big( |p| \, |q_n| - \be \, |q_n| - \be_0\big),
\]
which implies $|p| > \be$.
\end{proof}

\smallskip

\begin{Lemma}\label{graziella}
The function $L_\be^*$ is continuous on $\T^N \times \bP(\T^N) \times \ov{B(0,\be)}$.
\end{Lemma}

\begin{proof}
Fix $(x,\mu,p) \in \T^N \times \bP(\T^N) \times \ov{B(0,\be)}$ and a sequence $(x_n,\mu_n,p_n)$ with $(x_n,\mu_n,p_n) \to (x,\mu,p)$.

By Lemma \ref{grazia}, since $|p| \leq \be$, the supremum in the definition of $L^*_\be(x,\mu,p)$ is attained at some $q_0$, and since $ L_\be$ is continuous in all arguments, we have
\begin{eqnarray*}
  L^*_\be(x,\mu,p)
  &=& p \cdot q_0 - L_\be(x,\mu,q_0) \\
  &=& \lim_n \big( p_n \cdot q_0 -  L_\be(x_n,\mu_n,q_0) \big) \\
  &\leq& \liminf_n  L_\be^*(x_n,\mu_n,p_n),
\end{eqnarray*}
which yields lower semicontinuity at $(x,\mu,p)$.

Again by Lemma \ref{grazia}, for each $n$ there exists $q_n$ with
\[
  L^*_\be(x_n,\mu_n,p_n) = p_n \cdot q_n -  L_\be(x_n,\mu_n,q_n)
  \quad \text{for any $n$.}
\]
Assume by contradiction that $|q_n| \to + \infty$, up to subsequences. Then for $n$ large enough,
\[
  L^*_\be(x_n,\mu_n,p_n) = p_n \cdot q_n - \be \, |q_n| - \be_0,
\]
hence $q_n$ satisfies
\[
  \frac{q_n}{|q_n|}=\frac{p_n}{|p_n|},
\]
since with such a choice the first term in the right--hand side of the above formula is maximized and the other ones are not affected. Then
\[
  L^*_\be(x_n,\mu_n,p_n) =  |q_n| |p_n| - \be \, |q_n| - \be_0,
\]
which yields
\[
  L^*_\be(x_n,\mu_n,p_n) \leq - \be_0.
\]
By \eqref{beta00},
\[
  L^*_\be(x_n,\mu_n,p_n) < - L(x_n,\mu_n,0) = -L_\be(x_n,\mu_n,0),
\]
which contradicts the definition of $L_\be^*$.

We conclude that the sequence $q_n$ is bounded, hence, up to subsequences, $q_n \to q_1 \in \R^N$. By continuity of $L_\be$,
\[
  \lim_n  L^*_\be(x_n,\mu_n,p_n)
  = \lim_n \big( p_n \cdot q_n -  L_\be(x_n,\mu_n,q_n)\big)
  = p \cdot q_1 - L_\be(x,\mu,q_1)
  \leq L^*_\be(x,\mu,p).
\]
Combining upper and lower semicontinuity, we deduce the claimed continuity.
\end{proof}

\smallskip

As a consequence of the above lemmata, we have
\begin{equation}\label{grazia1}
  L^{**}_\be(x, \mu,q) =  \max_{p \in \ov{B(0,\be)}} \big( p \cdot q - L^*_\be(x,\mu,p) \big)
  \quad \text{for any $(x,\mu,q)$.}
\end{equation}

\smallskip

\begin{proof}[{\bf Proof of Proposition \ref{graziolla}}]
Let $(x_0,\mu_0,q_0) \in \T^N \times \bP(\T^N) \times \R^N$ and let $(x_n,\mu_n,q_n)$ be a sequence converging to $(x_0,\mu_0,q_0)$. By \eqref{grazia1} and the continuity of
\[
  p \mapsto p \cdot q -  L^*_\be(x,\mu,p),
\]
there exist, by Lemma \ref{graziella}, $p_0$, $p_n$ in $\ov{B(0,\be)}$ with
\begin{eqnarray*}
  L^{**}_\be(x_0,\mu_0,q_0) &=& p_0 \cdot q_0 - L^*_\be(x_0,\mu_0,p_0), \\
  L^{**}_\be(x_n,\mu_n,q_n) &=& p_n \cdot q_n - L^*_\be(x_n,\mu_n,p_n) \quad\hbox{for any $n$}.
\end{eqnarray*}
We derive
\begin{eqnarray*}
   L^{**}_\be(x_0,\mu_0,q_0)
   &=& p_0 \cdot q_0 -  L_\be^*(x_0,\mu_0,p_0) \\
   &=& \lim_n \big( p_0\cdot q_n -  L_\be^*(x_n, \mu_n,p_0) \big) \\
   &\leq& \liminf_n  L^{**}_\be(x_n,\mu_n,q_n),
\end{eqnarray*}
since $p_n$ converges, up to subsequences, to some $p_1 \in \ov{B(0,\be)}$, we get
\begin{eqnarray*}
 \lim_n  L^{**}_\be(x_n,\mu_n,q_n)
 &=& \lim_n  \big( p_n \cdot q_n -  L_\be^*(x_n,\mu_n,p_n)\big) \\
 &=&  p_1\cdot q_0 - L_\be^*(x_0, \mu_0,p_1) \\
 &\leq&   L^{**}_\be(x_0,\mu_0,q_0).
\end{eqnarray*}
The two above inequalities prove the assertion.
\end{proof}

\bigskip

 \bibliography{MFG}
\bibliographystyle{siam}

\end{document}